\documentclass[a4paper,12pt,oneside,reqno]{amsart}
\usepackage[latin1]{inputenc}
\usepackage[english]{babel}
\usepackage[dvips]{graphicx}
\usepackage{amssymb}
\usepackage{amsmath}
\usepackage{amsthm}
\usepackage{url}
\usepackage{a4wide}

\usepackage[hmargin=1.8cm,vmargin=2.3cm]{geometry} 
\usepackage[foot]{amsaddr}

\def\[#1\]{\begin{align*}#1\end{align*}}
\def\be#1\ee{\begin{align}#1\end{align}}
\def\bea#1\eea{\begin{align}#1\end{align}}
\def\ben#1\een{\begin{align*}#1\end{align*}}

\newcommand{\diff}{d}

\newcommand{\n}[1]{\left\Vert #1\right\Vert}
\newcommand{\la}{\left\langle}
\newcommand{\ra}{\right\rangle}
\newcommand{\R}{\mathbb{R}}

\newcommand{\N}{\mathbb{N}}

\newcommand{\mc}[1]{\mathcal{#1}}
\newcommand{\s}{\mc}
\newcommand{\p}[2]{\frac{\partial #1}{\partial #2}}
\newcommand{\q}[2]{\frac{\partial^2 #1}{\partial #2^2}}
\newcommand{\pa}[1]{{{\partial}\over{\partial #1}}}
\newcommand{\dif}[1]{\frac{\diff}{\diff #1}}
\def\ol#1{\overline{#1}}
\newcommand{\longto}{\mathop{\longrightarrow}}

\newcommand{\diam}{\mathrm{diam}}

\def\[#1\]{\begin{align*}#1\end{align*}}
\def\be#1\ee{\begin{align}#1\end{align}}
\def\bea#1\eea{\begin{align}#1\end{align}}
\def\ben#1\een{\begin{align*}#1\end{align*}}

\newtheoremstyle{theorem}{0.5cm}{0.5cm}%
   {}
   {}
   {\bfseries}
   {}
   {2ex}
   {\thmname{#1}\thmnumber{ #2}\thmnote{ #3}}
\theoremstyle{theorem}
\newtheorem{theorem}{Theorem}[section]

\newtheorem{proposition}[theorem]{Proposition}
\newtheorem{example}[theorem]{Example}
\newtheorem{corollary}[theorem]{Corollary}
\newtheorem{remark}[theorem]{Remark}

\newtheorem{lemma}[theorem]{Lemma}

\begin{document}

\title{On Existence of Spatially Regular Strong Solutions for a Class of Transport Equations}
 
\author{Jouko Tervo}
\author{Petri Kokkonen$^1$}
\address{$^1$University of Eastern Finland,
Department of Technical Physics, Kuopio, Finland}
\email{$^1$petri.kokkonen@uef.fi}

\maketitle

\begin{abstract}
The paper considers existence of spatially regular solutions for a class of  
linear Boltzmann transport equations.
The related transport problem is an (initial) inflow boundary value problem. 
This problem is characteristic with variable multiplicity, that is, the rank of 
the boundary matrix (here a scalar) is not constant on the boundary.
It is known that for these types of (initial) boundary value problems
the full higher order Sobolev regularity cannot generally be established.
In this paper we present  Sobolev regularity results for solutions of
linear Boltzmann transport problems
when the data belongs to appropriate anisotropic Sobolev spaces
whose elements are zero on the inflow and characteristic
parts of the boundary.
\end{abstract}

\section{Introduction}\label{intro}

Various linear and non-linear transport equations have been widely used as basic  models e.g. in fluid mechanics and particle physics. Usually the solution must obey certain initial and boundary value conditions.
It is known that the solutions of these types of problems have poor or limited (Sobolev) regularity properties.
These regularity issues may be subtle even for solutions
of a pure time-dependent advection equation ${\p {u}{t}}+c\cdot\nabla_xu=f$.

We consider regularity of solutions for certain classes of
linear {\it continuous slowing down Boltzmann transport equations} of the form
\be\label{i1}
a(x,E){\p \psi{E}}(x,\omega,E)
+\omega\cdot\nabla_x\psi(x,\omega,E)+\Sigma(x,\omega,E)\psi(x,\omega,E)-(K_{r}\psi)(x,\omega,E)=f(x,\omega,E)
\ee
for $(x,\omega,E)\in G\times S\times I$.
The cases $a=0$ and $a\neq 0$ will be studied separately.

Here $G\subset\R^3$ is the spatial domain,  $S\subset\R^3$ is the unit sphere (velocity direction domain) and $I=[E_0,E_m]$ is the energy interval.
In addition, the solution $\psi$ satisfies the inflow boundary condition
\be\label{i2}
{\psi}_{|\Gamma_-}=g,
\ee
where $\Gamma_-$ is the \emph{inflow boundary}
$\{(y,\omega,E)\in (\partial G)\times S\times I\ |\ \omega\cdot\nu(y)<0\}$ (see Section \ref{prem} below)
and the initial condition (if $a\neq 0$)
\be\label{i1-i}
\psi(\cdot,\cdot,E_m)=0,
\ee
where $E_m$ is the so-called {\it cut-off energy}.
The conditions (\ref{i2}) and (\ref{i1-i}) guarantee that the overall initial inflow boundary value problem (\ref{i1}), (\ref{i2}), (\ref{i1-i}) is (under relevant assumptions)  well-posed.

The \emph{restricted collision operator} $K_r$ is a partial integral type operator
(for some details we refer to \cite[section 5]{tervo18-up}) of the form
\begin{multline}\label{i2-k}
K_r\psi=
\int_{S'}\int_{I'}\sigma^1(x,\omega',\omega,E',E)\psi(x,\omega',E') dE' d\omega'
+
\int_{ S'}\sigma^2(x,\omega',\omega,E)\psi(x,\omega',E) d\omega'
\\
+
\int_{I'}\int_{0}^{2\pi}
\sigma^3(x,E',E)
\psi(x,\gamma(E',E,\omega)(s),E')ds dE'=:K_r^1\psi+K_r^2\psi+K_r^3\psi.
\quad
\end{multline}
For clarity we denote $S=S'$ and $I=I'$ when the integration variables
are $\omega'$ and $E'$, respectively.
In this paper, we confine ourselves mainly to the case where $K_r=K_r^2$.
The function $f=f(x,\omega,E)$ in (\ref{i1}) represents the internal source
and $g=g(y,\omega,E)$ in \eqref{i2} represents the inflow boundary source.
The solution $\psi$ of the problem \eqref{i1}, \eqref{i2}, \eqref{i1-i} 
describes the fluence of the considered particle type. 

We assume throughout that the spatial dimension $n=3$ ($G\subset\R^3$)
even though this is by no means essential.
In addition, we restrict to the case where the Lebesgue index $p=2$
in $L^p$-based spaces.
The regularity results are formulated by using certain anisotropic Sobolev spaces 
$H^{(m_1,m_2,m_3)}(G\times S\times I)$ (section \ref{fs}) which are relevant in formulations of such results for transport equations
because they allow for different degrees of regularity
with respect to the variables $x$, $\omega$ and $E$.
In this paper, we consider only the regularity with respect to the spatial 
variable $x$ and so $m_2=m_3=0$.

The equation (\ref{i1}) is a special case of the equation
\be\label{i1-a}
a(x,E){\p \psi{E}}
+c(x,E)\Delta_S\psi
+d(x,\omega,E)\cdot\nabla_S\psi
+\omega\cdot\nabla_x\psi+\Sigma(x,\omega,E)\psi-K_{r}\psi=f,
\ee
where $\Delta_S$ is the Laplace-Beltrami operator and $\nabla_S$ is the gradient operator on $S$.
In \cite{tervo19} we have given some reasons to use an equations like (\ref{i1-a}) for modelling charged particle transport
(for example for the propagation of electrons and positrons) in dose calculation for radiation therapy.  
The starting point is that the differential cross sections for charged particles may contain hyper-singularities with respect to energy 
variable and hence the corresponding exact (original) collision operator is  a partial {\it hyper-singular integral operator}.
This singular  integral operator can be reasonably approximated by partial integro-differential operator which leads to an equation of the form (\ref{i1-a}). 
In practice the realistic dose calculation model consists of the coupled system of equations like (\ref{i1-a}) since the simultaneous evolution 
of several species of particles (electrons, positrons, photons) must be modelled.
  
The existence and uniqueness of solutions for various classes of (initial) inflow boundary value problems like \eqref{i1}, \eqref{i2}, \eqref{i1-i} is well understood 
(e.g. \cite{dautraylionsv6}, \cite{agoshkov}, \cite{frank10}, \cite{tervo18-up}, 
 \cite{tervo18}). The regularity issues of solutions are much more poorly understood. 
A well-known difficulty in the initial boundary  value theory of first-order PDEs is that the equation may be characteristic on the boundary (the so-called boundary matrix has a non-zero kernel).
Especially the problems in which the rank of the boundary matrix is not constant on the boundary
(the so-called problems with \emph{variable multiplicity}) are challenging
(e.g.  \cite{rauch85}, \cite{nishitani96}, \cite{nishitani98},  \cite{takayama02}).
The problem considered in this paper is with variable multiplicity. The other challenge in the present case is the inclusion of the restricted collision operator $K_r$ which is a (non-local) partial integral operator.
The equation in the problem is not a pure PDE.
The full higher order Sobolev regularity can not be expected for the characteristic  initial boundary value problems.
However, partial (more precisely, tangential) regularity can be proved for problems with constant multiplicity (\cite{rauch85}, \cite{morando09}). For problems with variable multiplicity even this partial (tangential) regularity need not be valid.

The application of general theory of  initial boundary  value problems related to the systems of first-order PDEs yields some foundations and results for the regularity of solutions of transport problems.
For certain problems with variable multiplicity higher order  regularity can be shown when the data is appropriately restricted.
One such result can be found in \cite{nishitani96}, the idea of which we shall utilize in the section \ref{nis} below.
In certain weighted co-normal Sobolev spaces the higher order (tangential) regularity can be achieved with less restrictive data (see e.g. \cite{nishitani98}).

There are also some regularity results specific to initial inflow boundary value transport problems. We mention some of them.
In \cite[Chapter 4]{agoshkov} regularity results for solutions
are presented for certain mono-kinetic Boltzmann transport equations (BTEs).
Some collision operators obey the so-called \emph{smoothing property} (\cite{golse}).
In \cite{chen20} one has applied the  smoothing property to derive regularity results.
An iteration formula together with the smoothing property of the collision operator has been used there to prove the regularity of solutions
with respect to the spatial variables up to order $1-\epsilon$ for any $0<\epsilon<1$ for certain transport problems.
In \cite{guo17} one has obtained the first order ($W^{p,1}$ or weighted $W^{p,1}$) regularity results for some non-linear time-dependent BTEs.  
A similar kind of non-linear equation along with its linearized equation (section 4.2 therein)
have been studied in \cite{alexandre} in the global case $G=\R^3$.
There one has obtained, under certain assumptions, weighted regularity of solutions with respect to all variables.
In \cite{chen23} one has applied Neumann series to find spatially $H^1$-regular solutions when the diameter of the domain $G$ is small enough.
In \cite{tervo19-b} we obtained some regularity results
for the BTEs like considered here in the global case $G=\R^3$. Therein we applied 
explicit formulas for solutions of certain types of equations
and the difference method for more general equations.
We remark that in the global case $G=\R^3$ the regularity issues are less challenging and
the "regularity of solutions may increase according to the data". 
When $G\neq \R^3$ 
the corresponding  initial inflow boundary value problem (\ref{i1}), (\ref{i2}), (\ref{i1-i})  is characteristic with variable multiplicity (as mentioned above) which leads to limited regularity of solutions.
In fact, it is known that the regularity of the solutions of the transport equations
with respect to (the spatial) $x$-variable, for example,
is in general limited up to the space
$H^{(s_1,0,0)}(G\times S\times I)$ with $s_1<3/2$ regardless of the smoothness of the data 
(see the counterexample given in \cite{tervo17-up}, section 7.1).
In \cite{tervo23}  we studied regularity of solutions in bounded domains by using various methods. At first we applied the well-known explicit formulas of solutions for equations without the restricted collision operator $K_r$.  
After that,
in section 5  we considered  regularity of solutions   for an equation  involving 
the integral operator  $K_r$.
The techniques were based on Neumann series.
Furthermore, the treatments therein were confined to special cases of transport equations (\ref{i1}) and the data ($f,\ g$) was appropriately restricted. 

This paper is organized as follows. After preliminaries (section \ref{prem}),
we show in section \ref{nis} results of existence of solutions in spaces $H_0^{(m,0,0)}(G\times S\times I,\Gamma_-)$ which are completions of
\[
\bigcap_{k=1}^\infty
\{ \psi\in C_0^k(\ol G, C( S\times I))\ |
\ {\rm supp}(\psi)\cap (\Gamma_- \cup\Gamma_0)=\emptyset \}
\] 
with respect to $H^{(m,0,0)}(G\times S\times I)$-norms.
The formulations are related to those in \cite{nishitani96}.
At first, the existence of $H_0^{(m,0,0)}(G\times S\times I,\Gamma_-)$-solutions is shown in the case $a(x,E)=0$.
The proof is based on the $m$-dissipativity of the smallest closed extension (in 
the spaces $H_0^{(m,0,0)}(G\times S\times I,\Gamma_-)$) of the transport operator. 
Then the existence of $C^1(I,H_0^{(m,0)}(G\times S,\Gamma_-'))$-solutions
is deduced for the complete equation (\ref{i1}) by applying the theory of 
evolution equations
(the space $H_0^{(m,0)}(G\times S,\Gamma'_-)$ is defined similarly
to $H_0^{(m,0,0)}(G\times S\times I,\Gamma_-)$).
We remark that in these formulations the \emph{data $f$ must 
be restricted to the space} $H_0^{(m,0,0)}(G\times S\times I,\Gamma_-)$. Moreover, the 
inflow boundary data $g=0$ due to the definition of the space $H_0^{(m,0,0)}(G\times S\times I,\Gamma_-)$. 
Finally, the appendix (section 4) contains some additional proofs.

\section{Preliminaries}\label{prem}

We assume that $G$ is an open bounded set in $\R^3$ whose boundary $\partial G$  belongs to $C^{\infty}$
in the sense that $\ol{G}$ is a $C^\infty$-submanifold of $\R^3$ with boundary.
(see e.g. \cite[Section 1.2]{grisvard}, \cite[Chapter 1]{lee:2013}).
Moreover, we assume that the closure $\ol{G}$ of $G$ is \emph{strictly convex}
(consequently $G$ is convex).

The unit outward pointing normal on $\partial G$ is denoted by $\nu$. The surface measure on $\partial G$ is $d\sigma$.
Let $S=S_2$ be the unit sphere in $\R^3$ equipped with the usual surface measure $d\omega$.
Furthermore, let $I$ be an interval $[E_0,E_m]$  of $\R$  where
$E_0\geq 0, \ E_0<E_m<\infty$. The interior $]E_0,E_m[$ is denoted by $I^\circ$.
We also denote
\[
\Gamma:=\partial G\times S\times I,
\quad
\Gamma':=\partial G\times S,
\]
and
\[
\Gamma_{-}:=\{(y,\omega,E)\in \partial G\times S\times I\ |\ \omega\cdot\nu(y)<0\},
\quad
\Gamma'_{-}:=\{(y,\omega)\in \partial G\times S\ |\ \omega\cdot\nu(y)<0\}
\]
\[
\Gamma_{+}:=\{(y,\omega,E)\in \partial G\times S\times I\ |\ \omega\cdot\nu(y)>0\},
\quad
\Gamma'_{+}:=\{(y,\omega)\in \partial G\times S\ |\ \omega\cdot\nu(y)>0\}.
\] 
\[
\Gamma_{0}:=\{(y,\omega,E)\in \partial G\times S\times I\ |\ \omega\cdot\nu(y)=0\},
\quad
\Gamma'_{0}:=\{(y,\omega)\in \partial G\times S\ |\ \omega\cdot\nu(y)=0\}.
\] 

For  $(x,\omega)\in G\times  S$ the {\it escape time mapping $t(x,\omega)$ } is defined by
\be\label{def:t:1}
t(x,\omega)&=\inf\{s>0\ |\ x-s\omega\not\in G\}\\ \nonumber
&=\sup\{t>0\ |\ x-s\omega\in G\ {\rm for\ all}\ 0< s <t\}.
\ee
Furthermore, for $(y,\omega)\in \Gamma'_-\cup\Gamma'_+$ the {\it escape-time mapping $\tau_\pm(y,\omega)$ (from boundary to boundary) } is defined by
\bea\label{def:tau_pm:1}
&
\tau_-(y,\omega):=
\inf\{s>0\ |\ y+s\omega\not\in G\},\quad (y,\omega)\in \Gamma'_-\\
&
\tau_+(y,\omega):=
\inf\{s>0\ |\ y-s\omega\not\in G\},\quad (y,\omega)\in \Gamma'_+.
\nonumber
\eea
We find that (see e.g. Lemma \ref{le:basic_geometric_properties:1} below for why $\tau_{\pm}(y,\omega)>0$)
\[
0<t(x,\omega)\leq d:=\diam(G),\quad 0<\tau_{\pm}(y,\omega)\leq d,
\]
where $\diam(G)$ is the diameter of $G$.

\subsection{Basic Function Spaces}\label{fs}

All function spaces below are real valued.
Define the  space $W^2(G\times S\times I)$  by
\bea\label{fseq1}
W^2(G\times S\times I)
=\{\psi\in L^2(G\times S\times I)\ |\  \omega\cdot\nabla_x \psi\in L^2(G\times S\times I) \}.
\eea
The space $W^2(G\times S\times I)$ is equipped  with the inner product
\be\label{fs4}
\la {\psi},v\ra_{W^2(G\times S\times I)}=\la {\psi},v\ra_{L^2(G\times S\times I)}+
\la\omega\cdot\nabla_x\psi,\omega\cdot\nabla_x v\ra_{L^2(G\times S\times I)}.
\ee
Then $W^2(G\times S\times I)$  is a Hilbert space.
Let for $k\in\N\cup\{\infty\}$
\begin{align}\label{eq:def:C^(k,0,0):1}
C^{(k,0,0)}(G\times S\times I^\circ)
:=\{{}& f:G\times S\times I^\circ\to\R
\ |\ \textrm{all partial derivatives} \\
{}& \textrm{$(x,\omega,E)\mapsto \partial^\alpha_x f(x,\omega,E)$, $|\alpha|\leq k$, exist and} \nonumber \\
{}& \textrm{are continuous on $G\times S\times I^\circ$}\} \nonumber \\
C^{(k,0,0)}_0(\R^3\times S\times\R)
:=\{{}& f\in C^{(k,0,0)}(\R^3\times S\times \R)\ |\ \textrm{support of $f$ is compact in $\R^3\times S\times \R$}\}
\nonumber
\end{align}
and
\begin{align}\label{eq:def:C^k_ol_G:1}
C^k(\ol G\times S\times I):={}& \{f_{|G\times S\times I^\circ}\ |\ f\in C_0^k(\R^3\times S\times\R)\} \\
C^{(k,0,0)}(\ol G\times S\times I):={}& \{f_{|G\times S\times I^\circ}\ |\ f\in C_0^{(k,0,0)}(\R^3\times S\times\R)\}
\nonumber
\end{align}
The space $C^1(\ol G\times S\times I)$ is a dense subspace of $W^2(G\times S\times I)$.
Clearly, $C^k(\ol G\times S\times I)\subset C^{(k,0,0)}(\ol G\times S\times I)$ for all $k$.

The space of $L^2$-functions on $\Gamma_-$ with respect to the measure
$|\omega\cdot\nu|\, d\sigma d\omega dE$  is denoted by $T^2(\Gamma_-)$
that is, $T^2(\Gamma_-)=L^2(\Gamma_-,|\omega\cdot\nu|\, d\sigma d\omega dE)$.
$T^2(\Gamma_-)$   is a Hilbert space and its natural inner product is
\be\label{fs8}
\la g_1,g_2\ra_{T^2(\Gamma_-)}=\int_{\Gamma_-}g_1(y,\omega,E)g_2(y,\omega,E)|\omega\cdot\nu|\, d\sigma(y) d\omega dE.
\ee
In patches $\Gamma_-^i=\{(y,\omega,E)\in (\partial G)_i\times S\times I\ |\ \omega\cdot\nu(y)<0\}$ the inner product \eqref{fs8} is computed by
\[
\la g_1,g_2\ra_{T^2(\Gamma_-^i)}
=\int_{S\times I}\int_{U_{i,-,\omega}}g_1(h_i(z),\omega,E)g_2(h_i(z),\omega,E) |\omega\cdot\nu(h_i(z))| \n{(\partial_{z_1}h_i\times\partial_{z_2}h_i)(z)} dz d\omega dE
\]
where $U_{i,-,\omega}:=\{z\in U_i|\ \omega\cdot\nu(h_i(z))<0\}$ and $h_i:= \varphi_i^{-1}$ (a local parametrization)
when $(U_i,\varphi_i)$ is a chart of $(\partial G)_i$.

The spaces $T^2(\Gamma_+)$ and $T^2(\Gamma)$ and their inner products are similarly defined.
In addition, we define the (Hilbert) space
\[
T_{ \tau_{\pm}}^2(\Gamma_{\pm})=L^2(\Gamma_{\pm},\tau_{\pm}(\cdot,\cdot)|\omega\cdot\nu|\ d\sigma d\omega dE)
\]
where the canonical inner product
\[
\la g_1,g_2\ra_{T_{\tau_{\pm}}^2(\Gamma_\pm)}=\int_{\Gamma_\pm}g_1(y,\omega,E)g_2(y,\omega,E) \tau_{\pm}(y,\omega)|\omega\cdot\nu|\ d\sigma(y) d\omega dE
\]
is used.

Since $\tau_{\pm}(y,\omega)\leq \diam(G)=d<\infty$ we obtain
from \cite{dautraylionsv6}, p. 252 
or \cite{tervo18-up}, Theorem 2.16 

\begin{theorem}\label{tth}
The inflow trace operators 
\[
\gamma_{\pm}:W^2(G\times S\times I)\to T_{\tau_{\pm}}^2(\Gamma_{\pm}) \ {\rm defined\ by}\ \gamma_{\pm}(\psi):=\psi_{|\Gamma_{\pm}}
\] 
are (well-defined) bounded surjective operators and they have bounded right inverses $L_\pm:T^2_{\tau_\pm}(\Gamma_\pm)\to { W}^2(G\times S\times I)$ that is, $\gamma_\pm\circ L_\pm=\s I$ (the identity map). The operators $L_\pm$  are called {\it lifts}.
\end{theorem}

\begin{remark}\label{lift}
The lift $L_-$ can be given explicitly, for instance by
\[
(L_-g)(x,\omega,E)=e^{-\lambda t(x,\omega)}g(x-t(x,\omega)\omega,\omega,E)
\]
where $\lambda\geq 0$. For $\lambda=0$ the  lift $L_-$ is an isometrics $L_-:T^2_{\tau_-}(\Gamma_-)\to { W}^2(G\times S\times I)$. Note that the lift $L_-$ is not unique. Analogous observations are valid for $L_+$.
\end{remark}

From Theorem \ref{tth} it follows   that
the trace operators $\gamma_\pm:W^2(G\times S\times I)\to \ L^2_{\rm loc}(\Gamma_\pm,
|\omega\cdot\nu|\ d\sigma d\omega dE)$ are continuous. 
However, the trace $\gamma(\psi),\ \psi\in { W}^2(G\times S\times I)$ is not necessarily in the space $T^2(\Gamma)$.  Hence we define the spaces
\be\label{fs12}
\widetilde{ W}^2(G\times S\times I)=\{\psi\in { W}^2(G\times S\times I)\ |\  \gamma(\psi)\in T^2(\Gamma) \}.
\ee
The space $\widetilde { W}^2(G\times S\times I)$ is equipped with the inner product
\be\label{fs14}
\la \psi,v\ra_{\widetilde { W}^2(G\times S\times I)}=\la \psi,v\ra_{{ W}^2(G\times S\times I)}+ \la \gamma(\psi),\gamma(v)\ra_{T^2(\Gamma)}.
\ee
Then $\widetilde { W}^2(G\times S\times I)$ is a Hilbert space (\cite{tervo17-up}, Proposition 2.5).
For convex domains $G$ 
the space $\widetilde { W}^2(G\times S\times I)$ is the completion of $C^1(\ol G\times S\times I)$ with respect to the inner product (\ref{fs14}) and then $C^1(\ol G\times S\times I)$ is dense in
$\widetilde { W}^2(G\times S\times I)$ (\cite{tervo18-up}, Corollary 2.22).

For $v\in \widetilde { W}^2( G\times S\times I)$ and $\psi\in\widetilde { W}^2(G\times S\times I)$ it holds the Green's formula (Stokes's Theorem)
\begin{align}\label{green}
\int_{G\times S\times I}(\omega\cdot \nabla_x \psi)v\ dxd\omega dE
+\int_{G\times S\times I}(\omega\cdot \nabla_x v)\psi\ dxd\omega dE=
\int_{\partial G\times S\times I}(\omega\cdot \nu) v\ \psi\ d\sigma d\omega dE
\end{align}
which is obtained by classical Stokes Theorem  for $v,\ \psi\in C^1(\ol G\times S\times  I)$ and then by the density arguments for general $v\in \widetilde { W}^2(G\times S\times I)$ and $\psi\in\widetilde { W}^2(G\times S\times I)$.

\subsection{Anisotropic Sobolev Spaces}\label{an-spaces}

Let $m=(m_1,m_2,m_3)\in \N_0^3$ be a multi-index (for simplicity we restrict ourselves here to integer multi-indexes but the definitions below can be generalized for fractional indexes $s=(s_1,s_2,s_3)\in [0,\infty[^3$ in the standard way).
Define an {\it  anisotropic (or mixed-norm)  Sobolev spaces} $H^m(G\times S\times I):=H^{m,2}(G\times S\times I)$ by 
\bea
&
H^{m}(G\times S\times I)
:=\Big\{\psi\in L^2(G\times S\times I)\ |\ \partial_x^\alpha\partial_{{\omega}}^\beta\partial_E^l\psi \in L^2(G\times S\times I),\nonumber\\
&
 {\rm for\ all}\ |\alpha|\leq m_1,\ |\beta|\leq m_2,\ l\leq m_3\Big\}.
\eea
Here we denote more shortly $\partial_x^\alpha:={{\partial^\alpha}\over{\partial x^\alpha}}$
and similarly $\partial_\omega^\beta:={{\partial^\beta}\over{\partial \omega^\beta}}={{\partial^{\beta_1}}\over{\partial \omega^{\beta_1}}}{{\partial^{\beta_2}}\over{\partial \omega^{\beta_2}}}$
where $\{\partial_{\omega_j},\ j=1,2\}$ is a local  basis of the tangent space $T(S)$. 
We recall that for sufficiently smooth functions $f:S\to\R$ 
\[
{\p f{\omega_j}}_{\Big|\omega}={\partial\over{\partial w_j}}(f\circ h)_{\Big|z=h^{-1}(\omega)},\ j=1,2
\]
where  $h:W\to S\setminus S_0$ is a parametrization of  $S\setminus S_0$ ($S_0$ has the surface measure zero).
Moreover, recall that $f\in L^2(S)$  if and only if $f\circ h\in L^2(W,\n{\partial_1h\times\partial_2h}dz)$. 
For $f\in L^2(S)$ we define the inner product
\be\label{sn}
\la f_1,f_2\ra_{L^2(S)}=\int_Sf_1 f_2 d\omega=\int_W(f_1\circ h)(f_2\circ h)\n{\partial_1h\times\partial_2h}du.
\ee

The space $H^{m}(G\times S\times I)$  is a Hilbert space when equipped with the  inner product
\be\label{hminner}
\la\psi,v\ra_{H^{m}(G\times S\times I)}
:=\sum_{|\alpha|\leq m_1,|\beta|\leq m_2,
l\leq m_3}
\la\partial_x^\alpha\partial_{\omega}^\beta\partial_E^l\psi,
\partial_x^\alpha\partial_{\omega}^\beta\partial_E^l v\ra_{L^2(G\times S\times I)}.
\ee
The corresponding norm is
\[
\n{\psi}_{H^{m}(G\times S\times I)}=\Big(\sum_{|\alpha|\leq m_1}\sum_{|\beta|\leq m_2}\sum_{l\leq m_3}
\n{\partial_x^\alpha\partial_{\omega}^\beta\partial_E^l\psi }_{L^2(G\times S\times I)}^2\Big)^{{1\over 2}}.
\]

Note that for $m'\geq m$ (this means that $m_j'\geq m_j$ for $j=1,2,3$)
\[
H^{m'}(G\times S\times I)\subset H^{m}(G\times S\times I)
\]
and that (since $G\times S\times I$ is bounded)
\[
W^{\infty,m}(G\times S\times I)\subset
H^{m}(G\times S\times I),
\]
where for $m=(m_1,m_2,m_3)\in\N_0^3$
\bea\label{def:W^(infty,m):1}
W^{\infty,m}(G\times S\times I)
:=\Big\{\psi\in L^\infty(G\times S\times I)
\ \big|\ &
\n{\partial_x^\alpha\partial_{\omega}^\beta\partial_E^l\psi}_{L^\infty(G\times S\times I)}<\infty \\
&
{\rm for }\ 
|\alpha|\leq m_1,\ |\beta|\leq m_2,\  l\leq m_3\Big\}
\nonumber
\eea
equipped with the natural norm.

\vskip1.5cm

\section{Existence of Solutions in Spaces $H_0^{(m,0,0)}(G\times S\times I,\Gamma_-)$}\label{nis}
 
Define for $m\in\N_0\cup\{\infty\}$ the spaces
\begin{align}\label{eq:def:C^k_ol_G_Gamma_-:1}
C_0^{(m,0,0)}(\ol G\times S\times I,\Gamma_-)
:={}& \{ \psi\in C^{(m,0,0)}(\ol G\times S\times I)\ |\ {\rm supp}(\psi)\cap (\Gamma_- \cup\Gamma_0)=\emptyset\}
\end{align}
Note that for $\psi\in C_0^{(m,0,0)}(\ol G\times S\times I,\Gamma_-)$
\be\label{d-pos}
{\rm dist}({\rm supp}(\psi),\Gamma_-\cup \Gamma_0)>0
\ee
since ${\rm supp}(\psi)$ and $\Gamma_-\cup\Gamma_0=\{(y,\omega,E)\ |\ \omega\cdot\nu(y)\leq 0\}=\ol\Gamma_-$ are compact disjoint subsets of $\ol G\times S\times I$ for $\psi\in C^{(m,0,0)}_0(\ol G\times S\times I,\Gamma_-)$.

Furthermore, let $H_0^{(m,0,0)}(G\times S\times I,\Gamma_-)$, $m\in\N_0$, be the completion of
$\cap_{k=1}^\infty C_0^{(k,0,0)}(\ol G\times S\times I,\Gamma_-)$ with respect to the inner product
\[
\la\psi,v\ra_{H^{(m,0,0)}(G\times S\times I)}=\sum_{|\alpha|\leq m}\la \partial_x^\alpha\psi,
 \partial_x^\alpha v\ra_{L^2(G\times S\times I)}. 
\]

Using standard arguments, one can show the following  (cf. \cite{nishitani96})

\begin{proposition}\label{pr:H^(m,0,0,Gamma_-):characterization:1}
If $\psi\in H^{(m,0,0)}(G\times S\times I)$ is such that
\[
{\rm supp}(\psi)\cap (\Gamma_-\cup\Gamma_0)=\emptyset,
\]
then $\psi\in H_0^{(m,0,0)}(G\times S\times I,\Gamma_-)$.
\end{proposition}

From the inflow Trace Theorem (\cite[Theorem 2.15]{tervo18-up}) it follows that $\psi_{|\Gamma_-}=0$ for any $\psi\in H_0^{(m,0,0)}(G\times S\times I,\Gamma_-),\ m\geq 1$.
Define for $\psi\in \cap_{k=1}^\infty C_0^{(k,0,0)}(\ol G\times S\times I,\Gamma_-)$ a boundary norm 
\[
\n{\gamma(\psi)}_{H_0^{(m,0,0)}(\Gamma,\Gamma_-)}^2=\sum_{|\alpha|\leq m}\n{\partial_x^\alpha\psi}_{T^2(\Gamma)}^2
=\sum_{|\alpha|\leq m}\n{\partial_x^\alpha\psi}_{T^2(\Gamma_+)}^2
\]
and let $H_0^{(m,0,0)}(\Gamma,\Gamma_-)$ be the completion of $\{\gamma(\psi)|\ \psi\in \cap_{k=1}^\infty C_0^{(k,0,0)}(\ol G\times S\times I,\Gamma_-)\}$ with respect to $\n{\cdot }_{H_0^{(m,0,0)}(\Gamma,\Gamma_-)}$-norm.

In the sequel we shall
consider the regularity of solutions of the transport equation 
\[
T\psi:=a{\p \psi{E}}
+\omega\cdot\nabla_x\psi+\Sigma\psi-K_{r}\psi=f
\] 
using scales $H_0^{(m,0,0)}(G\times S\times I,\Gamma_-)$.

\subsection{Higher-order Regularity Results for Convection-Attenuation Equation}\label{conv-eq}

Firstly, consider the case where $a=K_r=0$.
Denote
\[
P\psi:=\omega\cdot\nabla_x\psi+\Sigma\psi.
\]

\subsubsection{Auxiliary Accretivity Results}\label{nis-sub-a}

The space $C^{(\infty,0,0)}(\ol G\times S\times I)$ is not dense in $W^{\infty,(m,0,0)}(G\times S\times I)$.
Hence for our purposes, we need to consider a subspace of $W^{\infty,(m,0,0)}(G\times S\times I)$ in which
the former space is dense.

Let $\s W^{\infty,(m,0,0)}(G\times S\times I)$ be the completion of $C^{(\infty,0,0)}(\ol G\times S\times I)$ with respect to $\n{\cdot}_{W^{\infty,(m,0,0)}(G\times S\times I)}$.
Since for $\phi\in C^{(\infty,0,0)}(\ol G\times S\times I)$ and $\psi\in \cap_{k=1}^\infty C_0^{(k,0,0)}(\ol G\times S\times I,\Gamma_-)$
the product $\phi\psi\in \cap_{k=1}^\infty C_0^{(k,0,0)}(\ol G\times S\times I,\Gamma_-)$, we find that for $\Sigma\in  W^{\infty,(m,0,0)}(\ol G\times S\times I)$ and $\psi\in \cap_{k=1}^\infty C_0^{(k,0,0)}(\ol G\times S\times I,\Gamma_-)$
the product $\Sigma\psi\in H_0^{(m,0,0)}(\ol G\times S\times I,\Gamma_-)$ and
\be\label{li}
\n{\Sigma\psi}_{H_0^{(m,0,0)}(G\times S\times I,\Gamma_-)}\leq c(m)\n{\Sigma}_{W^{\infty,(m,0,0)}(\ol G\times S\times I)}\n{\psi}_{H_0^{(m,0,0)}(G\times S\times I,\Gamma_-)}
\ee
for some constant $0<c(m)<\infty$ depending on the order $m$.
Hence $\Sigma:H_0^{(m,0,0)}(G\times S\times I,\Gamma_-)\to H_0^{(m,0,0)}(G\times S\times I,\Gamma_-)$ is bounded
as a linear operator $\psi\mapsto \Sigma\psi$.

\begin{remark}
We imposed the requirement $\Sigma\in \s W^{\infty,(m,0,0)}(G\times S\times I)$ only for technical reasons. The below regularity results are valid assuming only that $\Sigma\in  W^{\infty,(m,0,0)}(G\times S\times I)$.
\end{remark}

Suppose that $\Sigma\in \s  W^{\infty,(m,0,0)}(G\times S\times I)$.
Since ${\rm supp}(P\psi)\subset {\rm supp}(\psi)$ we find that $P\psi\in H_0^{(m,0,0)}(\ol G\times S\times I,\Gamma_-)$
when $\psi\in \cap_{k=1}^\infty C_0^{(k,0,0)}(\ol G\times S\times I,\Gamma_-)$. 
Define an unbounded operator $P_m: H_0^{(m,0,0)}(G\times S\times I,\Gamma_-)\to  H_0^{(m,0,0)}(G\times S\times I,\Gamma_-)$ by
\[
\begin{cases}
& D(P_m):=\cap_{k=1}^\infty C_0^{(k,0,0)}(G\times S\times I,\Gamma_-),\\
&
P_m\psi=P\psi.
\end{cases}
\] 
Let $\widetilde P_m:H_0^{(m,0,0)}(G\times S\times I,\Gamma_-)\to H_0^{(m,0,0)}(G\times S\times I,\Gamma_-)$
be the smallest closed extension of $P_m$ (i.e. the closure of $P_m$), that is,
\[
\begin{cases}
D(\widetilde P_m) &= \{\psi\in H_0^{(m,0,0)}(G\times S\times I,\Gamma_-)\ |\ \textrm{there exists a sequence $\{\psi_n\}\subset D(P_m)$}\ \\
&
\phantom{=\{}
\textrm{and $f\in H_0^{(m,0,0)}(G\times S\times I,\Gamma_-)$ such that}\ 
\n{\psi_n-\psi}_{H_0^{(m,0,0)}(G\times S\times I,\Gamma_-)}\to 0\ \\
&
\phantom{=\{}
\textrm{and $\n{P_m\psi_n-f}_{H_0^{(m,0,0)}(G\times S\times I,\Gamma_-)}\to 0$ as $n\to\infty$} \\
\widetilde P_m\psi &= f.
\end{cases}
\]

The operator $\widetilde P_m$ is densely defined and closed as an unbounded operator $H_0^{(m,0,0)}(G\times S\times I,\Gamma_-)\to H_0^{(m,0,0)}(G\times S\times I,\Gamma_-)$.
In addition, we have

\begin{lemma}\label{nis-le1}
Suppose that $\Sigma\in \s W^{\infty,(m,0,0)}(G\times S\times I)$
and let $c(m)$ be a positive constant so that \eqref{li} holds
and write $C':=c(m)\n{\Sigma}_{W^{\infty,(m,0,0)}(G\times S\times I)}$.
Then for $C\geq C'$
\be\label{nis-1}
\la (\widetilde P_m+C\s I)\psi,\psi\ra_{H_0^{(m,0,0)}(G\times S\times I,\Gamma_-)}\geq (C-C')\n{\psi}_{H_0^{(m,0,0)}(G\times S\times I,\Gamma_-)}^2,
\quad \psi\in D(\widetilde P_m).
\ee
In particular,
$\widetilde P_m+C\s{I}:H_0^{(m,0,0)}(G\times S\times I,\Gamma_-)\to H_0^{(m,0,0)}(G\times S\times I,\Gamma_-)$  is an accretive operator
for $C\geq C'$.

\end{lemma}

\begin{proof}
 
Using  Green's formula (\ref{green}) we obtain for $\psi,\ v\in \cap_{k=1}^\infty C_0^{(k,0,0)}(\ol G\times S\times I,\Gamma_-)$
\bea\label{trunc-10}
&
\la\omega\cdot\nabla_x\psi,v\ra_{H_0^{(m,0,0)}(G\times S\times I,\Gamma_-)}
=\sum_{|\alpha|\leq m}\la \partial_x^\alpha(\omega\cdot\nabla_x\psi),\partial_x^\alpha v\ra_{L^2(G\times S\times I)}\nonumber\\
&
=\sum_{|\alpha|\leq m}\la \omega\cdot\nabla_x(\partial_x^\alpha\psi),\partial_x^\alpha v\ra_{L^2(G\times S\times I)}\nonumber\\
&
=
-\sum_{|\alpha|\leq m}\la \partial_x^\alpha\psi,\omega\cdot\nabla_x (\partial_x^\alpha v)\ra_{L^2(G\times S\times I)}
+\sum_{|\alpha|\leq m}\int_{\partial G\times S\times I}(\partial_x^\alpha\psi) (\partial_x^\alpha v) (\omega\cdot\nu) d\sigma d\omega dE  
\nonumber\\
&
=
-\sum_{|\alpha|\leq m}\la \partial_x^\alpha\psi, \partial_x^\alpha (\omega\cdot\nabla_x v)\ra_{L^2(G\times S\times I)}
+\sum_{|\alpha|\leq m}\int_{\Gamma_+}(\partial_x^\alpha\psi) (\partial_x^\alpha v)(\omega\cdot\nu) d\sigma d\omega dE \\
&
=
-\la\psi,\omega\cdot\nabla_x v\ra_{H_0^{(m,0,0)}(G\times S\times I,\Gamma_-)}
+\sum_{|\alpha|\leq m}\int_{\Gamma_+}(\partial_x^\alpha\psi) (\partial_x^\alpha v)(\omega\cdot\nu) d\sigma d\omega dE
\eea
since $\partial_x^\alpha\psi=\partial_x^\alpha v=0$ on $\Gamma_-$ and $\omega\cdot\nu=0$ on $\Gamma_0$.
Hence
\[
\la\omega\cdot\nabla_x\psi,\psi\ra_{H_0^{(m,0,0)}(G\times S\times I,\Gamma_-)}
={1\over 2}\sum_{|\alpha|\leq m}\int_{\Gamma_+}(\partial_x^\alpha\psi) (\partial_x^\alpha v)(\omega\cdot\nu) d\sigma d\omega dE 
={1\over 2}\n{\psi}_{H^{(m,0,0)}_0(\Gamma,\Gamma_-)}^2
\]
and so by \eqref{li}, we have
\bea\label{nis-2}
&
\la (P_m+C\s{I})\psi,\psi\ra_{H_0^{(m,0,0)}(G\times S\times I,\Gamma_-)}\nonumber\\
&
=
\la\omega\cdot\nabla_x\psi,\psi\ra_{ H_0^{(m,0,0)}(G\times S\times I,\Gamma_-)}
+\la (C+\Sigma)\psi,\psi\ra_{H_0^{(m,0,0)}(G\times S\times I,\Gamma_-)}
\nonumber \\
&
\geq 
{1\over 2}\n{\psi}_{H_0^{(m,0,0)}(\Gamma,\Gamma_-)}^2
+(C-c(m)\n{\Sigma}_{W^{\infty,(m,0,0)}(G\times S\times I)})\n{\psi}_{H_0^{(m,0,0)}(G\times S\times I,\Gamma_-)}^2
\nonumber \\
&
=
{1\over 2}\n{\psi}_{H_0^{(m,0,0)}(\Gamma,\Gamma_-)}^2
+(C-C')\n{\psi}_{H_0^{(m,0,0)}(G\times S\times I,\Gamma_-)}^2
\eea
which gives the estimate (\ref{nis-1}) due to the definition of $\widetilde P_m$.
In virtue of (\ref{nis-1})
\[
\la (\widetilde P_m+C\s I)\psi,\psi\ra_{H_0^{(m,0,0)}(G\times S\times I,\Gamma_-)}\geq 0,\quad \psi\in D(\widetilde P_m)
\]
when $C\geq C'$
and so $\widetilde P_m+C\s I$ is accretive in  $H_0^{(m,0,0)}(G\times S\times I,\Gamma_-)$.
\end{proof}

Let $C'$ be as in Lemma \ref{nis-le1} and let $C>C'$, $c:=C-C'$. Then by Lemma \ref{nis-le1}
\be\label{nis-1-c}
\la (\widetilde P_m+C\s I)\psi,\psi\ra_{H_0^{(m,0,0)}(G\times S\times I,\Gamma_-)}\geq 
c\n{\psi}_{H_0^{(m,0,0)}(G\times S\times I,\Gamma_-)}^2,\ 
\psi\in D(\widetilde P_m).
\ee
which implies that the range $R(\widetilde P_m+C\s I)$ is closed in $H_0^{(m,0,0)}(G\times S\times I,\Gamma_-)$.

Recall that for $f\in L^2(G\times S\times I)$ the unique solution $\psi\in W^2(G\times S\times I)$ of the problem
\[
\omega\cdot\nabla_x\psi+\Sigma\psi=f,\ \psi_{|\Gamma_-}=0
\]
is
\be\label{nis-3}
\psi=\int_0^{t(x,\omega)}e^{-\int_0^t\Sigma(x-s\omega,\omega,E)ds}f(x-t\omega,\omega,E)dt.
\ee
For the escape time mapping $t$ we define the following extension $\tilde t:\ol G\times S\to\R$
\be\label{olt}
\tilde t(x,\omega):=
\begin{cases}
t(x,\omega),\ &(x,\omega)\in G\times S\\
0,\ &(x,\omega)\in \Gamma'_-\cup\Gamma'_0\\
\tau_+(x,\omega),\ &(x,\omega)\in\Gamma'_+.
\end{cases}
\ee
Defining $\tilde{t}$ in this way, i.e. extending the definition of $t$ from $G\times S$
onto $\ol{G}\times S$ the way we do,
will be justified by Proposition \ref{pr:continuity_of_tilde_t:1} given below.

We formulate the following geometric result.

\begin{lemma}\label{le:basic_geometric_properties:1}
The following geometric properties are valid:
\begin{itemize}

\item[(i)] If $(x,\omega)\in\Gamma'_+$,
there exists $s_0>0$ such that $x-s\omega\in G$ for all $0<s\leq s_0$
and $\tau_+(x,\omega)>0$.

\item[(ii)] If $(x,\omega)\in\Gamma'_-$,
there exists $s_0>0$ such that $x+s\omega\in G$ for all $0<s\leq s_0$
and $\tau_-(x,\omega)>0$.

\item[(iii)] For every $x\in\partial G$
it holds
\[
\nu(x)\cdot (z-x)<0\quad \forall z\in G.
\]

\item[(iv)]
Let $(x,\omega)\in (G\times S)\cup \Gamma'_+$.
Then $\tilde{t}(x,\omega)>0$ and for $y:=x-\tilde{t}(x,\omega)\omega$ it holds $(y,\omega)\in\Gamma'_-$.
\end{itemize}
\end{lemma}

\begin{proof}
See Appendix, Section \ref{app:le:basic_geometric_properties:1:proof}.
\end{proof}

\begin{remark}\label{nonc}
The mapping $\tilde t$ defined by \eqref{olt} is not necessarily continuous on $\ol{G}\times S$ for convex domains $G$.
For example, let $G=\R_+^3=\{x\in\R^3\ |\ x_3>0\}$. Then for $(x,\omega')\in G\times S$, $\omega_3'> 0$,
\[
t(x,\omega')=\big|-{{x_3}\over{\omega_3'}}\big|.
\]
We find that for $(y,\omega)\in \Gamma_0'$ (for which $y_3=\omega_3=0$)
\[
\lim_{\overset{\scriptstyle (x,\omega')\to (y,\omega)}{(x,\omega')\in G\times S,\ x_3=\omega_3'}}t(x,\omega')=1,\
\lim_{\overset{\scriptstyle (x,\omega')\to (y,\omega)}{(x,\omega')\in G\times S,\ x_3=(\omega_3')^2}}t(x,\omega')=0
\]
and so the limit $\lim_{(x,\omega')\to (y,\omega),\ (x,\omega')\in G\times S} t(x,\omega')$ doest not exists.
Similar situation is valid for bounded smooth domain whose boundary contains a plane area.
\end{remark}

However, under our assumptions which include the strict convexity of $G$, we have
(see also Lemma 2.3.1 in \cite{anikonov:2002}).

\begin{proposition}\label{pr:continuity_of_tilde_t:1}
The function $\tilde{t}:\ol{G}\times S\to\R_+$ is continuous
and its restriction $t=\tilde{t}|_{G\times S}$ to $G\times S$ is $C^1$.
\end{proposition}

\begin{proof}
See Appendix, Section \ref{app:pr:continuity_of_tilde_t:1:proof}.
\end{proof}

\begin{example}\label{ex:pr:continuity_of_tilde_t:1:1}
Let $G=B(0,1)=\{x\in\R^3\ |\ \n{x}<1\}$ be the open ball in $\R^3$ centered at $0$ with unit radius.
Then $\ol{G}=\ol{B}(0,1)=\{x\in\R^3\ |\ \n{x}\leq 1\}$, $\partial G=\{x\in\R^3\ |\ \n{x}=1\}$
and $\ol{G}$ is strictly convex.

It turns out that in this case the extended escape time map $\tilde{t}:\ol{G}\times S\to\R_+$ is given by
\[
\tilde{t}(x,\omega)=x\cdot\omega+\sqrt{(x\cdot\omega)^2+1-\n{x}^2},
\quad (x,\omega)\in \ol{G}\times S
\]
and thus it is clearly continuous on $\ol{G}\times S$.

The spatial gradient of the escape time map $t=\tilde{t}|_{G\times S}$ is given by
\[
\nabla_x t(x,\omega)=\omega+\frac{(x\cdot\omega)\omega-x}{\sqrt{(x\cdot\omega)^2+1-\n{x}^2}},
\quad (x,\omega)\in G\times S
\]
and is clearly continuous on $G\times S$.
However, $\nabla_x t$ does \emph{not} extend continuously onto $\ol{G}\times S$.
\end{example}

\begin{lemma}\label{nis-le2}
Suppose that $f\in C(\ol G\times S\times I)$ such that ${\rm supp}(f)\cap (\Gamma_- \cup\Gamma_0)=\emptyset$.
Then for $\psi$ defined by (\ref{nis-3}) 
\be\label{is}
{\rm supp}(\psi)\cap (\Gamma_- \cup\Gamma_0)=\emptyset.
\ee
\end{lemma}

\begin{proof}
With a slight abuse of notation, we will write below $\tilde{t}(x,\omega,E)=\tilde{t}(x,\omega)$
whenever $(x,\omega,E)\in U$.

Let $\eta>0$ and let $U:=\ol{G}\times S\times I$.  Define (cf. \cite{nishitani96} where the weight function $\phi$ is in the role of $\tilde t(x,\omega)$) 
\[
\ol{U}_\eta:=\{z\in U\ |\ \tilde{t}(z)\geq \eta\}.
\]
Since $\tilde{t}$ is continuous\footnote{
The map $\ol{G}\times S\to\R_+$; $(x,\omega)\to \tilde{t}(x,\omega)$
is continuous and hence the map
$U=\ol{G}\times S\times I\to\R_+$; $(x,\omega,E)\mapsto \tilde{t}(x,\omega,E)=\tilde{t}(x,\omega)$
is continuous
}
by Proposition \ref{pr:continuity_of_tilde_t:1},
the set $\ol{U}_{\eta}$ is closed in $U$.
The assumption\linebreak
${\rm supp}(f)\cap (\Gamma_- \cup\Gamma_0)=\emptyset$ implies that there exists $\eta>0$ such that
\be\label{nis-10-d}
{\rm supp}(f)\subset \ol{U}_\eta.
\ee
This can be seen as follows.  Note that $\Gamma_-\cup\Gamma_0=\ol\Gamma_-$. Since by the assumption ${\rm supp}(f)\cap \ol\Gamma_- =\emptyset$ we have
$d':={\rm dist}({\rm supp}(f),\ol\Gamma_-)>0$. Choose $\eta:=d'$.
Then (\ref{nis-10-d}) holds: Let $z=(x,\omega,E)\in {\rm supp}(f)$.
Then ${\rm dist}(z,\ol\Gamma_-)\geq {\rm dist}({\rm supp}(f),\ol\Gamma_-)=\eta>0$
and hence
$z\in (\ol G\times S\times I)\setminus \ol\Gamma_-=(G\times S\times I)\cup\Gamma_+$.
Therefore, by (iv) of Lemma \ref{le:basic_geometric_properties:1}
the point $z$ is of the form $z=(y+s\omega,\omega,E)$ where $(y,\omega,E)\in \Gamma_-$ and $s>0$.
We find that $\tilde{t}(z)=\tilde t(y+s\omega,\omega,E)=s$ and
\[
\eta={\rm dist}({\rm supp}(f),\ol\Gamma_-)\leq {\rm dist}(z,\ol\Gamma_-)
\leq {\rm dist}(z,(y,\omega,E))
={\rm dist}((y+s\omega,\omega,E),(y,\omega,E)) 
=s
\]
and so $\tilde t(z)=s\geq\eta$, that is, $z\in \ol{U}_\eta$.

We verify that 
\be\label{nis-8-b}
{\rm supp}(\psi)\subset \ol U_\eta
\ee
which is equivalent to
\be\label{nis-9-c}
U\setminus \ol U_\eta \subset U\setminus{\rm supp}(\psi).
\ee
Let $z_0:=(x_0,\omega_0,E_0)\in U\setminus \ol U_\eta$. 
Since $ U\setminus \ol{U}_\eta$ is  open in $U$ there exists a  neighbourhood $V_{z_0}$ of $z_0$ such that $V_{z_0}\cap U \subset  U\setminus \ol{U}_\eta$.
Then $\tilde{t}(x,\omega)<\eta$ for all $(x,\omega,E)\in V_{z_0}\cap U$. We find that
for $0\leq t\leq \tilde t(x,\omega)$ 
\[
\tilde t(x-t\omega,\omega)=\tilde t(x,\omega)-t<\eta-t\leq\eta.
\]
Hence $(x-t\omega,\omega,E)\not\in \ol U_\eta$ for all $(x,\omega,E)\in V_{z_0}\cap U,\ 0\leq t\leq \tilde t(x,\omega)$ which implies by (\ref{nis-10-d}) that 
$(x-t\omega,\omega,E)\not\in {\rm supp}(f)$ for all $(x,\omega,E)\in V_{z_0}\cap U,\ 0\leq t\leq \tilde t(x,\omega) $. That is why, $f(x-t\omega,\omega,E)=0$ for $(x,\omega,E)\in V_{z_0}\cap U,\ 0\leq t\leq \tilde t(x,\omega)$ which then yields
\[
\psi(x,\omega,E)=\int_0^{\tilde{t}(x,\omega)}e^{-\int_0^t\Sigma(x-s\omega,\omega,E)ds}f(x-t\omega,\omega,E)dt=0,
\quad \forall (x,\omega,E)\in V_{z_0}\cap U.
\]
We conclude that   $z_0\not \in {\rm supp}(\psi)$ and so (\ref{nis-9-c}) holds. 

From (\ref{nis-8-b})
it immediately follows the assertion ${\rm supp}(\psi)\cap (\Gamma_- \cup\Gamma_0)=\emptyset$ since 
$\tilde{t}(x,\omega)=0$ on $\ol{\Gamma'}_-=\Gamma'_-\cup\Gamma'_0$. Namely, suppose that ${\rm supp}(\psi)\cap (\Gamma_- \cup\Gamma_0)\not=\emptyset$ and let $z=(x,\omega,E)\in {\rm supp}(\psi)\cap (\Gamma_- \cup\Gamma_0)$.
Then $\tilde{t}(x,\omega)=0$ (since $z\in \ol{\Gamma}_-$ i.e. $(x,\omega)\in \ol{\Gamma'}_-$).
On the other hand, by (\ref{nis-8-b}) it holds $\tilde{t}(x,\omega)\geq\eta>0$ (since $z\in {\rm supp}(\psi))$
which is a contradiction. Hence (\ref{is}) holds. This completes the proof.

\end{proof}

\begin{theorem}\label{nis-th1}
Suppose that $\Sigma\in \s W^{\infty,(m,0,0)}(G\times S\times I)$.
Let $C':= c(m)\n{\Sigma}_{W^{\infty,(m,0)}(G\times S\times I)}$, where $c(m)$ is as in Lemma \ref{nis-le1}. Then for every $C>C'$
\be\label{nis-5}
R(\widetilde P_m+C\s I)=H_0^{(m,0,0)}(G\times S\times I,\Gamma_-)
\ee
and so for any $f\in H_0^{(m,0,0)}(G\times S\times I,\Gamma_-)$ there exists
a unique strong solution $\psi\in D(\widetilde{P}_m)\subset H_0^{(m,0,0)}(G\times S\times I,\Gamma_-)$ of the problem
\be\label{nis-6}
\omega\cdot\nabla_x\psi+\Sigma\psi+C\psi=f,\quad \psi_{|\Gamma_-}=0.
\ee
\end{theorem}

\begin{proof}

By (\ref{nis-1}) the kernel $N(\widetilde P_m+C\s{I})=\{0\}$ and so the solution $\psi$, if it exists, is unique.
It suffices to show that (\ref{nis-5}) holds.
We show in Parts A and B below that
\be\label{nis-7}
\cap_{k=1}^\infty C_0^{(k,0,0)}(\ol G\times S\times I,\Gamma_-)\subset R(\widetilde P_m+C\s I)
\ee
which implies (\ref{nis-5}) since $\cap_{k=1}^\infty C_0^{(k,0,0)}(\ol G\times S\times I,\Gamma_-)$ is dense in $H_0^{(m,0,0)}(G\times S\times I,\Gamma_-)$ and since
the range $R(\widetilde P_m+C\s I)$ is closed in $H_0^{(m,0,0)}(G\times S\times I,\Gamma_-)$ (as mentioned above after Eq. \eqref{nis-1-c}).

\medskip
\noindent {\bf A.}
At first, assume that $\Sigma\in C^{(\infty,0,0)}(\ol G\times S\times I)$.
We prove that for any $f\in C_0^{(m,0,0)}(\ol G\times S\times I,\Gamma_-)$ the solution $\psi$ of
the problem $(P+C)\psi=f$, $\psi_{|\Gamma_-}=0$, that is,
\begin{align}\label{eq:th:nis-th1:psi:1}
\psi(x,\omega,E)=\int_0^{t(x,\omega)}e^{-\int_0^t(\Sigma+C)(x-s\omega,\omega,E)ds}f(x-t\omega,\omega,E)dt,
\quad (x,\omega,E)\in G\times S\times I
\end{align}
belongs to $C_0^{(m,0,0)}(\ol G\times S\times I,\Gamma_-)$.
This implies that for any $f\in  \cap_{k=1}^\infty C_0^{(k,0,0)}(\ol G\times S\times I,\Gamma_-)$
the solution $\psi\in \cap_{k=1}^\infty C_0^{(k,0,0)}(\ol G\times S\times I,\Gamma_-)$ and hence
\eqref{nis-7} holds. We proceed by induction with respect to $m$.

\medskip
\noindent {\bf A.1.}
Consider the first order derivatives  $\partial^\alpha_x\psi,\ |\alpha|=1$. Suppose that
$f\in C_0^{(1,0,0)}(\ol G\times S\times I,\Gamma_-)$. Recall that the partial derivatives ${\p t{x_j}}(x,\omega)$ exist and they are continuous in $G\times S$ since $G$ is convex (\cite[Proposition 4.7]{tervo17-up}).
We have for $(x,\omega,E)\in G\times S\times I$ and $j=1,2,3$,

\bea\label{hreg2}
{\p {\psi}{x_j}}(x,\omega,E)
&=\int_0^{t(x,\omega)}\big(-\int_0^t{\p \Sigma{x_j}}(x-s\omega,\omega,E)ds\big)
\, e^{-\int_0^t(\Sigma+C)(x-s\omega,\omega,E)ds}
\, f(x-t\omega,\omega,E)dt\nonumber\\
&
+\int_0^{t(x,\omega)} e^{-\int_0^t(\Sigma+C)(x-s\omega,\omega,E)ds}\, {\p {f}{x_j}}(x-t\omega,\omega,E)dt
\nonumber\\
&
+
e^{-\int_0^{t(x,\omega)}(\Sigma+C)(x-s\omega,\omega,E)ds}\, f(x-t(x,\omega)\omega,\omega,E)\, {\p t{x_j}}(x,\omega)\nonumber\\
&
=:h_{1}+h_{2}+h_{3}.
\eea
Since $f\in C_0^{(1,0,0)}(\ol G\times S\times I,\Gamma_-)$ by assumption
and $(x-t(x,\omega)\omega,\omega,E)\in\Gamma_-$ by (iv) of Lemma \ref{le:basic_geometric_properties:1},
the last term $h_{3}=0$ and so
\bea\label{hreg2-a}
{\p {\psi}{x_j}}(x,\omega,E)
&=\int_0^{t(x,\omega)}\big(-\int_0^t{\p \Sigma{x_j}}(x-s\omega,\omega,E)ds\big)
\, e^{-\int_0^t(\Sigma+C)(x-s\omega,\omega,E)ds}
\, f(x-t\omega,\omega,E)dt\nonumber\\
&
+\int_0^{t(x,\omega)} e^{-\int_0^t(\Sigma+C)(x-s\omega,\omega,E)ds}\, {\p {f}{x_j}}(x-t\omega,\omega,E)dt
\eea
for all $(x,\omega,E)\in G\times S\times I$ and $j=1,2,3$.
For strictly convex domain $G$ the extension of the escape time mapping $\tilde{t}$ (defined by (\ref{olt})) is continuous in $\ol{G}\times S$
(see Proposition \ref{pr:continuity_of_tilde_t:1})
and therefore we deduce from \eqref{hreg2-a} that
the partial derivatives ${\p {\psi}{x_j}},\ j=1,2,3$, belong to $C(\ol G\times S\times I)$.

By Lemma \ref{nis-le2}, it holds ${\rm supp}(\psi)\cap (\Gamma_-\cup\Gamma_0)=\emptyset$
(since ${\rm supp}(f)\cap (\Gamma_-\cup\Gamma_0)=\emptyset$).
Hence $\psi\in C_0^{(1,0,0)}(\ol G\times S\times I,\Gamma_-)$ and so the claim holds with $m=1$.

\medskip
\noindent {\bf A.2.}
Suppose that the claim holds for some $m\in\N$
Let $f\in C_0^{(m+1,0,0)}(\ol G\times S\times I,\Gamma_-)$.
The induction hypothesis implies that $\psi\in C_0^{(m,0,0)}(\ol{G}\times S\times,\Gamma_-)$
and hence $(\partial_x^\alpha\psi)_{|\Gamma_-}=0$ for $|\alpha|\leq m$.
Furthermore, since $\psi$ satisfies $\omega\cdot\nabla_x \psi+(\Sigma+C)\psi=f$ on $G\times S\times I$,
its partial derivatives $\partial_x^\alpha\psi$ for $|\alpha|\leq m$ satisfy
\be\label{nis-23}
\omega\cdot\nabla_x(\partial_x^\alpha\psi)+(\Sigma+C)(\partial_x^\alpha\psi)
=
\partial_x^\alpha f
-\sum_{\beta<\alpha}{\alpha \choose\beta}(\partial_x^{\alpha-\beta}\Sigma)(\partial_x^\beta\psi)=:f_\alpha
\quad\textrm{and}\quad
(\partial_x^\alpha\psi)_{|\Gamma_-}=0,
\ee
where $f_\alpha\in C_0^{(1,0,0)}(\ol G\times S\times I,\Gamma_-)$ for each $|\alpha|\leq m$.
Due to Part A.1., the unique solution $\partial_x^\alpha\psi$ of (\ref{nis-23})
belongs to $C_0^{(1,0,0)}(\ol G\times S\times I,\Gamma_-)$ for $|\alpha|\leq m$
which then implies that $\psi\in C_0^{(m+1,0,0)}(\ol G\times S\times I,\Gamma_-)$. Hence the claim holds with $m+1$.

\medskip
\noindent {\bf B.}
More generally, let $\Sigma\in \s W^{\infty,(m,0,0)}(G\times S\times I)$
and let $\{\Sigma_n\}\subset C^{(\infty,0,0)}(\ol G\times S\times I)$ be a sequence
such that $\Sigma_n\to \Sigma$ in $W^{\infty,(m,0,0)}(G\times S\times I)$ as $n\to\infty$.
Then by Part A., for $f\in \cap_{k=1}^\infty C_0^{(k,0,0)}(\ol G\times S\times I,\Gamma_-)$
the functions (see \eqref{eq:th:nis-th1:psi:1})
\[
\psi_n(x,\omega,E):=\int_0^{t(x,\omega)}e^{-\int_0^t(\Sigma_n+C)(x-s\omega,\omega,E)ds}f(x-t\omega,\omega,E)dt,
\quad (x,\omega,E)\in G\times S\times I,
\quad n=1,2,\dots
\]
belong to $\cap_{k=1}^\infty C_0^{(k,0,0)}(\ol G\times S\times I,\Gamma_-) = D(P_m)$.

Moreover,
\be\label{nis-26}
\psi_n\longto_{n\to\infty} \psi \quad \textrm{in $H_0^{(m,0,0)}(\ol G\times S\times I,\Gamma_-)$}.
\ee
For $m=1$ the convergence \eqref{nis-26} can be seen from \eqref{hreg2-a}.
For general $m\geq 1$ it can be seen by using induction as follows.

We first notice that
by Part A.2., Eq. \eqref{nis-23} and by \eqref{eq:th:nis-th1:psi:1},
$\partial_x^\alpha\psi_n$ can be expressed as,
with $f_{n,\alpha}$ defined by \eqref{nis-23} after replacing $\Sigma$ by $\Sigma_n$,
\begin{align}\label{eq:th:nis-th1:partial_x_alpha_psi_n:2}
(\partial_x^\alpha \psi_n)(x,\omega,E)
={}& \int_0^{t(x,\omega)} e^{-\int_0^t(\Sigma_n+C)(x-s\omega,\omega,E)ds} f_{n,\alpha}(x-t\omega,\omega,E)dt
\end{align}
with
\[
f_{n,\alpha}(x,\omega,E) = \partial_x^\alpha f - \sum_{\beta<\alpha}{\alpha \choose\beta}(\partial_x^{\alpha-\beta}\Sigma_n)(\partial_x^\beta\psi_n)(x,\omega,E)
\]
for all $(x,\omega,E)\in G\times S\times I$, $n=1,2,\dots$ and all $|\alpha|\leq m$.

Suppose that $\psi_n\to \psi$ in $H_0^{(k,0,0)}(\ol G\times S\times I,\Gamma_-)$ as $n\to\infty$ for some $1\leq k\leq m-1$.
From the above, we know that this is true for $k=1$.

Let $|\alpha|\leq k+1$.
From the induction hypothesis, we deduce
that $\partial_x^\beta\psi_n\to \partial_x^\beta\psi$ in $L^2(G\times S\times I)$ as $n\to\infty$
for all $\beta<\alpha$ (since then $|\beta|\leq k$).
Because $\Sigma_n\to \Sigma$ in $W^{\infty,(m,0,0)}(G\times S\times I)$ when $n\to\infty$,
we have $\partial_x^{\alpha-\beta}\Sigma_n\to \partial_x^{\alpha-\beta}\Sigma$ in $L^\infty(G\times S\times I)$ as $n\to\infty$
also for all $\beta<\alpha$.
From this one deduces that
$f_{n,\alpha}\to f_{\alpha}$ in $L^2(G\times S\times I)$ as $n\to\infty$
for all $|\alpha|\leq k+1$, where $f_{\alpha}$ is defined by \eqref{nis-23}.

Noticing that
$e^{-\int_0^t(\Sigma_n+C)(x-s\omega,\omega,E)ds}\to e^{-\int_0^t(\Sigma+C)(x-s\omega,\omega,E)ds}$
in $L^\infty(G\times S\times I)$ as functions of $(x,\omega,E)$,
one then shows using \eqref{eq:th:nis-th1:partial_x_alpha_psi_n:2} that
$\partial_x^\alpha\psi_n\to \partial_x^\alpha\psi$ in $L^2(G\times S\times I)$ when $n\to\infty$
for all $|\alpha|\leq k+1$,
i.e. $\psi_n\to\psi$ in $H_0^{(k+1,0,0)}(\ol G\times S\times I,\Gamma_-)$.
Therefore, by induction, the convergence \eqref{nis-26} holds (for $k=m$).

Now that $\Sigma_n\to \Sigma$ in $W^{\infty,(m,0,0)}(G\times S\times I)$
and $\psi_n\to\psi$ in $H_0^{(m,0,0)}(\ol G\times S\times I,\Gamma_-)$
as $n\to\infty$,
one has
\[
(P_m+C\s I)\psi_n={}& \omega\cdot\nabla_x\psi_n+\Sigma\psi_n+C\psi_n
=\omega\cdot\nabla_x\psi_n+\Sigma_n\psi_n+C\psi_n+\Sigma\psi_n-\Sigma_n\psi_n
\\
={}& f+(\Sigma-\Sigma_n)\psi_n\longto_{n\to\infty} f+0=f
\quad \textrm{in $H_0^{(m,0,0)}(\ol G\times S\times I,\Gamma_-)$}.
\]
Since also $\psi_n\in D(P_m)$ for all $n$ (by the above),
we conclude that $\psi\in D(\widetilde P_m)$ and $(\widetilde P_m+C\s I)\psi=f$.
Hence \eqref{nis-7} holds and the proof is complete.

\end{proof}

\subsection{Higher-order Regularity Results for Convection-Scattering Equation}\label{conv-scat-eq}

Secondly, consider the transport operator of the form
\[
T\psi:=\omega\cdot\nabla_x\psi+\Sigma\psi-K_{r}\psi.
\] 
For simplicity, we assume that the restricted collision operator $K_r$ is of the form
\begin{align}\label{def:K_r^2:1}
(K_r\psi)(x,\omega,E)=(K_r^2\psi)(x,\omega,E):=\int_{S'}\sigma^2(x,\omega',\omega,E)\psi(x,\omega',E)d\omega'.
\end{align}

We find that in general $K_r\psi\not\in H_0^{(m,0,0)}(G\times S\times I,\Gamma_-)$ for
$\psi\in H_0^{(m,0,0)}(G\times S\times I,\Gamma_-)$.
Thus we need to make some additional assumptions on the cross-section $\sigma^2$.
Let 
\begin{multline*}
C_0^{(\infty,0,0)}(\ol G\times S\times I, L^1(S'),\Gamma_-)
:=\{\sigma^2\in C^\infty(\ol G\times S\times I,L^1(S'))\ |\ 
{\rm supp}(\sigma^2)\cap (\Gamma_- \cup\Gamma_0)=\emptyset\}.
\end{multline*}
Furthermore, let $\s W_0^{\infty,(m,0,0)}(G\times  S\times I,L^1(S'),\Gamma_-)$ be the completion of the space\linebreak
$C_0^{(\infty,0,0)}(\ol G\times S\times I,L^1(S'),\Gamma_-)\cap W^{\infty,(m,0,0)}(G\times S'\times I,L^1(S))$ with respect to the norm
\[
\n{\cdot}_{W^{\infty,(m,0,0)}(G\times S\times I,L^1(S'))}
+\n{\cdot}_{W^{\infty,(m,0,0)}(G\times S'\times I,L^1(S))}.
\]

\begin{remark}\label{H0}
We find that $\sigma^2\in \s W_0^{\infty,(m,0,0)}(G\times S\times I,L^1(S'),\Gamma_-)$ if 
\be\label{s2-h0}
\sigma^2\in \s W_0^{\infty,m}\big(G,L^\infty(S\times I,L^1(S'))\big)\cap
W^{\infty,m}\big(G,L^\infty(S'\times I,L^1(S))\big),
\ee
where $\s W_0^{\infty,m}\big(G,L^\infty(S\times I,L^1(S'))\big)$ is the completion of $C_0^{\infty}\big(G,L^\infty(S\times I,L^1(S'))\big)$ with respect to the norm $\n{\cdot}_{W^{\infty,m}(G,L^\infty(S\times I,L^1(S')))}$. Here for a Banach space $X$ by $W^{\infty,m}(G,X)$ we mean the $X$-valued $L^\infty$-based Sobolev space of order $m\in\N_0$.
The condition (\ref{s2-h0}) means that $(\partial_x^\alpha\sigma^2)_{\big|(\partial G)\times S'\times S\times I}=0$ for $|\alpha|\leq m-1$ in the Sobolev sense.
\end{remark}

We show that $K_r$ is a bounded operator $H^{(m,0,0)}(G\times S\times I)\to H_0^{(m,0,0)}(G\times S\times I,\Gamma_-)$.

\begin{theorem}\label{el-k-b}
Suppose that $\sigma^2:G\times S'\times S\times I\to\R_+$ is a non-negative 
Borel-measurable function such that
\be\label{nis-9}
\sigma^2\in \s  W_0^{\infty,(m,0,0)}(G\times S\times I,L^1(S'),\Gamma_-).
\ee
Then $K_r: H^{(m,0,0)}(G\times S\times I)\to H_0^{(m,0,0)}(G\times S\times I,\Gamma_-)$ is a bounded  operator.
\end{theorem}

\begin{proof}
\medskip
\noindent {\bf A.}
At first, let us verify that $K_r$ is a bounded operator $H^{(m,0,0)}(G\times S\times I)\to H^{(m,0,0)}(G\times S\times I)$.

Since
\[
\s W_0^{\infty,(m,0,0)}(G\times S'\times S\times I,\Gamma_-)\subset 
W^{\infty,(m,0,0)}(G\times S\times I,L^1(S'))\cap W^{\infty,(m,0,0)}(G\times S'\times I,L^1(S)),
\]
the assumption (\ref{nis-9}) implies that  for all $|\alpha|\leq m$
\bea\label{bound-b}
&
\int_{S'}|(\partial_x^\alpha\sigma^2)(x,\omega',\omega,E)|d\omega' \leq \n{\sigma^2}_{ W^{\infty,(m,0,0)}(G\times S\times I,L^1(S'))}\quad {\rm for\ a.e.}\ (x,\omega,E)\in G\times S\times I
\nonumber\\
&
\int_{S}|(\partial_x^\alpha\sigma^2)(x,\omega',\omega,E)|d\omega\leq \n{\sigma^2}_{ W^{\infty,(m,0,0)}(G\times S'\times I,L^1(S))}\quad {\rm for\ a.e.}\ (x,\omega',E)\in G\times S'\times I.
\eea
Furthermore, by the Leibniz's rule and Cauchy-Schwarz's inequality 
we have for $\psi\in C^{(\infty,0,0)}(\ol G\times S\times I)$
\bea\label{nis-10a}
&
|\partial_x^\alpha(K_r\psi)(x,\omega,E)|
=\big|\int_{S'}\partial_x^\alpha\big(\sigma^2(x,\omega',\omega,E)\psi(x,\omega',E)\big)d\omega'\big| \\
={}& \big|\int_{S'} \sum_{\beta\leq \alpha}{\alpha \choose\beta}(\partial_x^{\alpha-\beta}\sigma^2)(x,\omega',\omega,E)(\partial_x^\beta\psi)(x,\omega',E) d\omega'\big| \nonumber \\
\leq {}&
\sum_{\beta\leq \alpha}{\alpha \choose\beta}
\int_{S'}|(\partial_x^{\alpha-\beta}\sigma^2)(x,\omega',\omega,E)|^{1/2+1/2}|(\partial_x^\beta\psi)(x,\omega',E)| d\omega' 
\nonumber\\
\leq {}& 
\sum_{\beta\leq \alpha}{\alpha \choose\beta}
\Big(\int_{S'}|(\partial_x^{\alpha-\beta}\sigma^2)(x,\omega',\omega,E)|
d\omega' \Big)^{1/2}
\Big(\int_{S'}|(\partial_x^{\alpha-\beta}\sigma^2)(x,\omega',\omega,E)||(\partial_x^\beta\psi)(x,\omega',E)|^2
d\omega' \Big)^{1/2} \nonumber
\eea
for every $|\alpha|\leq m$ and for a.e. $(x,\omega,E)\in G\times S\times I$.
Hence using the inequality $\big(\sum_{i=1}^n a_i\big)^2\leq n\sum_{i=1}^n a_i^2$ holding for any $a_1,\dots,a_n\in\R,\ n\in\N$
and using \eqref{bound-b}, we find
\bea\label{nis-10}
&
\n{K_r\psi}_{H^{(m,0,0)}(G\times S\times I)}^2
=
\sum_{|\alpha|\leq m}\int_{G\times S\times I} |\partial_x^\alpha(K_r\psi)(x,\omega,E)|^2 dx d\omega dE
\nonumber\\
\leq {}&
C_m'\sum_{|\alpha|\leq m}\sum_{\beta\leq \alpha} {\alpha \choose\beta}^2
\int_{G\times S\times I} \Big(\int_{S'}|(\partial_x^{\alpha-\beta}\sigma^2)(x,\omega',\omega,E)|
d\omega' \Big)\nonumber\\
&
\cdot
\Big(\int_{S'}|(\partial_x^{\alpha-\beta}\sigma^2)(x,\omega',\omega,E)||(\partial_x^\beta\psi)(x,\omega',E)|^2
d\omega'\Big) dx d\omega dE
\nonumber\\
\leq {}&
C_m'\sum_{|\alpha|\leq m}\sum_{\beta\leq \alpha} {\alpha \choose\beta}^2
\n{\sigma^2}_{W^{\infty,(m,0,0)}(G\times S\times I,L^1(S'))}
\n{\sigma^2}_{W^{\infty,(m,0,0)}(G\times S'\times I,L^1(S))}\nonumber\\
&
\cdot \int_{G\times S'\times I} |(\partial_x^\beta\psi)(x,\omega',E)|^2 dx d\omega' dE
\nonumber\\
&
\leq 
C_m\n{\sigma^2}_{ W^{\infty,(m,0,0)}(G\times S\times I,L^1(S'))}
\n{\sigma^2}_{ W^{\infty,(m,0,0)}(G\times S'\times I,L^1(S))}
\n{\psi}_{H^{(m,0,0)}(G\times S\times I)}^2,
\eea
where the constants $0<C_m'<\infty$ and $0<C_m<\infty$ depend only on $m$ and where we noticed that
\[
&
\int_{G\times S\times I}
\int_{S'}|(\partial_x^{\alpha-\beta}\sigma^2)(x,\omega',\omega,E)||(\partial_x^\beta\psi)(x,\omega',E)|^2
d\omega'  dx d\omega dE \\
={}&
\Big(\int_{S}|(\partial_x^{\alpha-\beta}\sigma^2)(x,\omega',\omega,E)|d\omega\Big)
\Big(\int_{G\times S'\times I} |(\partial_x^\beta\psi)(x,\omega',E)|^2 dx d\omega' dE\Big).
\]
The estimate \eqref{nis-10} shows that $K_r:H^{(m,0,0)}(G\times S\times I)\to H^{(m,0,0)}(G\times S\times I)$ is bounded.

\medskip
\noindent {\bf B.} Next we will show that $K_r\psi\in H_0^{(m,0,0)}(G\times S\times I,\Gamma_-)$
for $\psi\in H^{(m,0,0)}(G\times S\times I)$ which proves the theorem.
Since $\sigma^2\in {\s W}_0^{\infty,(m,0,0)}(G\times S\times I,L^1(S'),\Gamma_-)$ there exists a sequence $\{\sigma^2_n\}\subset C_0^{(\infty,0,0)}(\ol G\times S\times I,L^1(S'),\Gamma_-)\cap W^{\infty,(m,0,0)}(G\times S'\times I,L^1(S))$ such that $\sigma^2_n\to\sigma^2$ simultaneously in
$W^{\infty,(m,0,0)}(G\times S\times I,L^1(S'))$ 
and in $W^{\infty,(m,0,0)}(G\times S'\times I,L^1(S))$.
Let for $n=1,2,\dots$
\[
(K_{r,n}\psi)(x,\omega,E):=\int_{S'}\sigma_n^2(x,\omega',\omega,E)\psi(x,\omega',E)d\omega',
\quad (x,\omega,E)\in G\times S\times I.
\]
Since $\sigma^2_n\in C_0^{(\infty,0,0)}(\ol G\times S\times I,L^1(S'),\Gamma_-)$ we find
that for any $\psi\in H^{(m,0,0)}(G\times S\times I)$
and all $n\in\N$,
\[
{\rm supp}(K_{r,n}\psi)\cap (\Gamma_-\cup\Gamma_0)=\emptyset
\]
which implies that $K_{r,n}\psi\in H_0^{(m,0,0)}(G\times S\times I,\Gamma_-)$
by Proposition \ref{pr:H^(m,0,0,Gamma_-):characterization:1}.
Due to the estimate \eqref{nis-10},
one has for any $\psi\in \psi\in H^{(m,0,0)}(G\times S\times I)$,
\bea\label{nis-11}
&
\n{K_{r,n}\psi-K_{r}\psi}_{H^{(m,0,0)}(G\times S\times I)}^2
\nonumber\\
\leq {}&
C_m\n{\sigma_n^2-\sigma^2}_{W^{\infty,(m,0,0)}(G\times S\times I,L^1(S'))}
\n{\sigma_n^2-\sigma^2}_{W^{\infty,(m,0,0)}(G\times S'\times I,L^1(S))}
\n{\psi}_{H^{(m,0,0)}(G\times S\times I)}^2 
\eea
and so $K_{r,n}\psi\to K_{r}\psi$ in $H^{(m,0,0)}(G\times S\times I)$ as $n\to\infty$.
Hence $K_{r}\psi\in H_0^{(m,0,0)}(G\times S\times I,\Gamma_-)$
for any $\psi\in H^{(m,0,0)}(G\times S\times I)$ as desired.
\end{proof}

Theorem \ref{el-k-b} implies that
$K_r$ yields a bounded operator $H_0^{(m,0,0)}(G\times S\times I,\Gamma_-)\to H_0^{(m,0,0)}(G\times S\times I,\Gamma_-)$ under the assumption \eqref{nis-9}.
Let $\n{K_r}_m$ be the corresponding norm of $K_r$.

Furthermore, let $Q_m:H_0^{(m,0,0)}(G\times S\times I,\Gamma_-)\to H_0^{(m,0,0)}(G\times S\times I,\Gamma_-)$ be
\[
Q_m:=P_m-K_r=\omega\cdot\nabla_x\psi+\Sigma\psi-K_r\psi,\quad D(Q_m):=D(P_m)
\]
and notice that the smallest closed extension $\widetilde Q_m$ of $Q_m$ in $H_0^{(m,0,0)}(G\times S\times I,\Gamma_-)$ is
\[
\widetilde Q_m=\widetilde {P_m-K_r}=\widetilde P_m-K_r,\quad D(\widetilde Q_m)=D(\widetilde P_m).
\]

We have

\begin{theorem}\label{nis-th2}
Suppose that 
\[
&
\Sigma\in \s W^{\infty,(m,0,0)}(G\times S\times I)
\nonumber\\
&
\sigma^2\in \s  W_0^{\infty,(m,0,0)}(G\times S\times I,L^1(S'),\Gamma_-).
\]
Let $C'':= c(m)\n{\Sigma}_{W^{\infty,(m,0,0)}(G\times S\times I)}+\n{K_r}_m$, 
where $c(m)$ is as in Lemma \ref{nis-le1}.
Then for every $C>C''$, it holds
\be\label{nis-13}
R(\widetilde Q_m+C\s I)=H_0^{(m,0,0)}(G\times S\times I,\Gamma_-)
\ee
and so for any $f\in H_0^{(m,0,0)}(G\times S\times I,\Gamma_-)$ there exists a unique strong solution
$\psi\in H_0^{(m,0,0)}(G\times S\times I,\Gamma_-)$ of the problem
\[
\omega\cdot\nabla_x\psi+\Sigma\psi-K_r\psi+C\psi=f,\quad \psi_{|\Gamma_-}=0.
\]
\end{theorem}

\begin{proof} 
Recall that a linear, densely defined, closed and accretive operator $A:X\to X$ is $m$-accretive if and only if
\be\label{m-d}
R(\lambda\s I+A)=X
\ee
for some (and hence all) $\lambda >0$ (see \cite[section 1.4]{pazy83}). 

Recall that $C'=c(m)\n{\Sigma}_{W^{\infty,(m,0,0)}(G\times S\times I)}$
in Lemma \ref{nis-le1}.
Let $C_1>C'$.
Due to Lemma \ref{nis-le1} and Theorem \ref{nis-th1} the operator $\widetilde P_m+C_1\s I:H_0^{(m,0,0)}(G\times S\times I,\Gamma_-)\to H_0^{(m,0,0)}(G\times S\times I,\Gamma_-)$ is $m$-accretive. 
Moreover, by Theorem \ref{el-k-b} the operator $K_r:H_0^{(m,0,0)}(G\times S\times I,\Gamma_-)\to H_0^{(m,0,0)}(G\times S\times I,\Gamma_-)$ is bounded with the norm $\n{K_r}_m$.  
Since
\[
\la -K_r\psi,\psi\ra_{H_0^{(m,0,0)}(G\times S\times I,\Gamma_-)}\geq -\n{K_r}_m\n{\psi}_{H_0^{(m,0,0)}(G\times S\times I,\Gamma_-)}^2
\]
we find that for $\lambda_1>\n{K_r}_m$ the operator $\lambda_1\s I-K_r:H_0^{(m,0,0)}(G\times S\times I,\Gamma_-)\to H_0^{(m,0,0)}(G\times S\times I,\Gamma_-)$ is a bounded accretive operator.  
Hence the operator
\[
(\widetilde P_m+C_1\s I)+(\lambda_1\s I-K_r)=\widetilde Q_m+(C_1+\lambda_1)\s I
\]
is $m$-accretive (cf. \cite[Chapter III]{engelnagel}, or \cite[section 5.1]{tervo17-up}).

Let $C>C''=C'+\n{K_r}_m$.
We may then choose some $C_1>C'$ and $\lambda_1>\n{K_r}_m$ such that
$\mu:=C-(C_1+\lambda_1)>0$.
Since $\widetilde Q_m+(C_1+\lambda_1)\s I$ is $m$-accretive by the above,
it follows from \eqref{m-d} that
\[
R(\widetilde Q_m+C\s I)
=R\big((\widetilde Q_m+(C_1+\lambda_1)\s I)+\mu\s I\big)
=H_0^{(m,0,0)}(G\times S\times I,\Gamma_-).
\]

The solution is unique since with $\epsilon:=C-C''>0$,
it holds
\[
\la (\widetilde Q_m+C\s I)\psi,\psi\ra_{H_0^{(m,0,0)}(G\times S\times I,\Gamma_-)}\geq \epsilon \n{\psi}_{H_0^{(m,0,0)}(G\times S\times I,\Gamma_-)}^2
\]
for every $\psi\in D(\widetilde Q_m)$.
This completes the proof.
\end{proof}

\begin{remark}\label{nis-co1}
Suppose that 
\[
&
\Sigma\in \s W^{\infty,(m,0,0)}(G\times S\times I) 
\nonumber\\
&
\sigma^2\in  \s W_0^{\infty,(m,0,0)}(G\times S\times I,L^1(S'),\Gamma_-).
\]
Let $C''= c(m)\n{\Sigma}_{W^{\infty,(m,0,0)}(G\times S\times I)}+\n{K_r}_m$ and let $C>C''$. 
Assume that $f\in H_0^{(m,0,0)}(G\times S\times I,\Gamma_-)$ and
$g\in T^2(\Gamma_-)$ are such that
\be\label{cc-a}
f-(\Sigma L_-g-K_r(L_-g)+C L_-g)\in H_0^{(m,0,0)}(G\times S\times I,\Gamma_-),
\ee
where $L_-g$ is the lift of $g$ given in Remark \ref{lift} above.
Then there exists a unique solution
\[
\psi\in H_0^{(m,0,0)}(G\times S\times I,\Gamma_-)+L_-g
\]
of the problem
\be\label{nis-14}
\omega\cdot\nabla_x\psi+\Sigma\psi-K_r\psi+C\psi=f,\quad \psi_{|\Gamma_-}=g.
\ee

\begin{proof}
Applying the change of unknown $u:=\psi-L_-g$ the problem becomes
(recall that\linebreak
${\omega\cdot \nabla_x(L_-g)=0}$)
\be\label{nis-14-b}
\omega\cdot\nabla_xu+\Sigma u-K_ru+Cu=f-(\Sigma L_-g-K_r(L_-g)+C L_-g)=:F,
\quad u_{|\Gamma_-}=0,
\ee
where $F\in H_0^{(m,0,0)}(G\times S\times I,\Gamma_-)$ by \eqref{cc-a}.
By Theorem \ref{nis-th2} there exists a unique solution  $u\in  H_0^{(m,0,0)}(G\times S\times I,\Gamma_-)$ of \eqref{nis-14-b} which implies the assertion since $\psi=u+L_-g$.
\end{proof}

We remark that even for smooth $g$ the lift $L_-g$ may lie only in $H^{(s,0,0)}(G\times S\times I)$ where $s<{3\over 2}$ and so  the regularity of  solutions $\psi$ of the (non-homogeneous) problem (\ref{nis-14}) is limited.
\end{remark}

\subsection{A Higher-order Regularity Result for CSDA-equation}\label{evcsd}

Consider the (continuous slowing down) transport operator of the form
\[
T\psi:=
a{\p \psi{E}}
+\omega\cdot\nabla_x\psi+\Sigma\psi-K_{r}\psi.
\] 

We proceed analogously to \cite[sections 3.3., 3.4.]{tervo17} or \cite[section 5.4.]{tervo18-up}.
Recall the following result from the theory of evolution equations.

\begin{theorem}\label{evoth}
Suppose that $X$ is a Hilbert space and that for any fixed $t\in [0,T]$
the operator $A(t):X\to X$ is linear and closed
with domain $D(A(t))\subset X$. In addition, we assume that the following conditions hold:
\begin{itemize}
\item[(i)] The domain $D:=D(A(t))$ is independent of $t$ and is a dense subspace of $X$.

\item[(ii)] The operator $A(t)$ is $m$-dissipative for any fixed $t\in [0,T]$ 

\item[(iii)] For every $u\in D$ the mapping $f_u:[0,T]\to X$ defined by $f_u(t):=A(t)u$
is in $C^1([0,T],X)$.

\item[(iv)] $f\in C^1([0,T],X)$ and $u_0\in D$.
\end{itemize}

Then the (evolution) equation
\be\label{ecsd5}
{\p {u}t}-A(t)u=f,\quad u(0)=u_0,
\ee
has a unique solution $u\in C([0,T],D)\cap C^1([0,T],X)$.  

\end{theorem}

\begin{proof}
See \cite[Theorem 4.5.3, pp. 89-106]{tanabe}, \cite[pp. 126-182]{pazy83} or \cite[pp. 477-496]{engelnagel}.
\end{proof}

Consider the initial inflow boundary value problem
\be\label{pr-f}
a{\p \psi{E}}+\omega\cdot\nabla_x\psi+\Sigma\psi-K_{r}\psi=f,
\quad \psi_{|\Gamma_-}=0,
\quad \psi(\cdot,\cdot,E_m)=0.
\ee 
Again we assume that the restricted collision operator is of the form $K_r=K_r^2$ (see \eqref{def:K_r^2:1}) i.e.
\be\label{ecsd1}
(K_r\psi)(x,\omega,E)=\int_{S'}\sigma^2(x,\omega',\omega,E)\psi(x,\omega',E)d\omega'.
\ee
Then $K_r\psi$ can be written as $K_r(E)\psi(E)$,
where $\psi(E)(x,\omega):=\psi(x,\omega,E)$ and where for any fixed $E\in I$
the linear operator $K_r(E)$ is defined by 
\be\label{ke}
(K_r(E)\varphi)(x,\omega):=\int_{S'} {\sigma^2}(x,\omega',\omega,E)\varphi(x,\omega') d\omega',
\quad \varphi\in L^2(G\times S).
\ee
For simplicity, we assume that $E_0=0$. Then the energy-interval $I=[0,E_m]$ remains unchanged under the change of variables $E\rightarrow E_m-E$.

We make the following change of variables and of the unknown function
\be\label{ecsd2} 
\tilde\psi(x,\omega,E):={}&\psi(x,\omega,E_m-E), \nonumber\\
\phi:={}&e^{CE}\tilde\psi,
\ee
and denote 
\bea
\hat a(x,E)={}& a(x,E_m-E),\ \nonumber\\
\widehat \Sigma(x,\omega,E)={}&\Sigma(x,\omega,E_m-E)\nonumber\\
\hat \sigma^2(x,\omega',\omega,E)={}&\sigma^2(x,\omega',\omega,E_m-E)\nonumber\\
\hat f(x,\omega,E)={}&f(x,\omega,E_m-E)\nonumber\\
(\widehat K_r\phi)(x,\omega,E)={}&\int_S \hat \sigma^2(x,\omega',\omega,E)\phi(x,\omega',E)d\omega'
\eea
with $\phi\in L^2(G\times S\times I)$ in the definition of $\widehat {K_r}$.
Making these changes, we find that the problem \eqref{pr-f}
is equivalent to
$\phi$ satisfying the equation
\bea\label{se1a}
&{\p {\phi}E}-{1\over{\hat a}}\omega\cdot\nabla_x\phi-C\phi
-{1\over{\hat a}}\widehat\Sigma\phi
+{1\over{\hat a}}\widehat K_r\phi
=-{1\over{\hat a}}e^{CE}\hat f,
\eea
on $G\times S\times I$
along with satisfying the following inflow boundary and initial value conditions
\begin{align}
\phi_{|\Gamma_-}=0, \label{se2a} \\
\phi(\cdot,\cdot,0)={}&0 \hspace{1.5cm} \textrm{(on $G\times S$)}. \label{se3a}
\end{align}

The spaces $C_0^{(m,0)}(\ol G\times S,\Gamma_-')$, 
$\s W^{\infty,(m,0)}(G\times S)$ and $H_0^{(m,0)}(G\times S,\Gamma_-')$
used below are 
defined similarly as the spaces
$C_0^{(m,0,0)}(G\times S\times I,\Gamma_-)$, $\s  W^{\infty,(m,0,0)}(G\times S\times I)$ and $H_0^{(m,0,0)}(G\times S\times I,\Gamma_-)$ above.
Denote
\[
\widehat\Sigma(E):=\widehat\Sigma(\cdot,\cdot,E),\quad \hat a(E):=\hat a(\cdot,E).
\]
Assume that
\bea
&
\Sigma\in  \s W^{\infty,(m,0,0)}( G\times S\times I)\label{csda-6}\\
&
a\in  \s W^{\infty,(m,0)}( G\times I)\label{csda-7}\\
&
-a(x,E)\geq \kappa>0\quad \forall (x,E)\in G\times I\label{csda-8}
\eea
for some constant $\kappa>0$.
Then we find that for
$\varphi\in H_0^{(m,0)}(G\times S,\Gamma_-')$
\be\label{csda-4}
\n{{1\over {\hat a(E)}}\widehat\Sigma(E)\varphi}_{H_0^{(m,0)}(G\times S,\Gamma_-')}\leq c(m)\n{{1\over {\hat a(E)}}\widehat \Sigma(E)}_{W^{\infty,(m,0)}( G\times S,\Gamma_-')}\n{\varphi}_{H_0^{(m,0)}(G\times S,\Gamma_-')},
\ee
where $c(m)$ is as in \eqref{li} and where
\be\label{csda-5}
\n{{1\over {\hat a(E)}}\widehat \Sigma(E)}_{W^{\infty,(m,0)}( G\times S)}\leq C_1\n{a}_{W^{\infty,(m,0)}( G\times I)}\n{\Sigma}_{W^{\infty,(m,0)}( G\times S\times I)}
\ee
for some constant $C_1$ which does not depend on $E$.
Hence the multiplication operators ${1\over {\hat a(E)}}\widehat \Sigma(E):H_0^{(m,0)}(G\times S,\Gamma_-')\to H_0^{(m,0)}(G\times S,\Gamma_-')$ are bounded for every $E\in I$
(and in fact they are uniformly bounded over $E\in I$).

The operators $\widehat K_r(E)$ satisfy the following:

\begin{lemma}\label{csda-th1}
Suppose that $a$ obeys (\ref{csda-7}) and (\ref{csda-8}) and that $\sigma^2:G\times S'\times S\times I\to\R$ is a  measurable function such that 
\be\label{cdsa-1}
\sigma^2\in \s W_0^{\infty,(m,0,0)}(G\times S\times I,L^1(S'),\Gamma_-).
\ee
Then ${1\over {\hat a(E)}}\widehat K_r(E): H_0^{(m,0)}(G\times S,\Gamma_-')\to H_0^{(m,0)}(G\times S,\Gamma_-')$ are bounded operators
for $E\in I$ and their norms $\n{{1\over {\hat a(E)}}\widehat K_r(E)}_m$ obey
\be\label{csda-2}
\n{{1\over {\hat a(E)}}\widehat K_r(E)}_m\leq
C_2\n{a}_{W^{\infty,(m,0)}( G\times I)}\n{\sigma^2}_{ W^{\infty,(m,0,0)}(G\times S\times I,L^1(S'))}
\n{\sigma^2}_{ W^{\infty,(m,0,0)}(G\times S'\times I,L^1(S))},
\ee
where $C_2$ does not depend on $E$.
\end{lemma}

\begin{proof}
The bound \eqref{csda-2} is derived in the same way as the bounds \eqref{nis-10}, \eqref{csda-4} and \eqref{csda-5}.
We omit further details.
\end{proof}

Let
\[  
Q_C(E)\varphi:={1\over {\hat a(E)}}\omega\cdot\nabla_x\varphi+{1\over {\hat a(E)}}\widehat \Sigma(E)\varphi+C\varphi-{1\over {\hat a(E)}}\widehat K_{r}(E)\varphi. 
\]
Under the above assumptions on $\Sigma$ and $\sigma^2$ we  define (for a fixed $E\in I$) an unbounded operator $Q_{C,m}(E): H_0^{(m,0)}(G\times S,\Gamma_-')\to  H_0^{(m,0)}(G\times S,\Gamma_-')$ by
\[
\begin{cases}
& D(Q_{C,m}(E)):=\cap_{k=1}^\infty C_0^{(k,0)}(G\times S,\Gamma_-') \\
&
Q_{C,m}(E)\varphi=Q_C(E)\varphi.
\end{cases}
\] 
Let $\widetilde Q_{C,m}(E):H_0^{(m,0)}(G\times S,\Gamma_-')\to H_0^{(m,0)}(G\times S,\Gamma_-')$ be the smallest closed extension (the closure) of $Q_{C,m}(E)$ that is,
\[
\begin{cases}
D(\widetilde Q_{C,m}(E)) &= \{\varphi\in H_0^{(m,0)}(G\times S,\Gamma_-')\ |\ \textrm{there\ exists\ a\ sequence $\{\varphi_n\}\subset D(Q_{C,m}(E))$}\ \\
& \phantom{=\{}
{\rm and}\ f\in H_0^{(m,0)}(G\times S,\Gamma_-')\ {\rm such\ that}\ 
\n{\varphi_n-\varphi}_{H_0^{(m,0)}(G\times S,\Gamma_-')}\to 0\ \\
& \phantom{=\{}
{\rm and}\ \n{Q_{C,m}(E)\varphi_n-f}_{H_0^{(m,0)}(G\times S,\Gamma_-')}\to 0\ {\rm as}\ n\to\infty\} \\
\widetilde Q_{C,m}(E)\varphi &= f.
\end{cases}
\]

Since by \eqref{csda-5}, \eqref{csda-2} the operator 
\[
R_{C,m}\varphi:={1\over {\hat a(E)}}\widehat \Sigma(E)\varphi+C\varphi-{1\over {\hat a(E)}}\widehat K_{r}(E)\varphi
\]
is a bounded operator $H_0^{(m,0)}(G\times S,\Gamma_-')\to H_0^{(m,0)}(G\times S,\Gamma_-')$ and since $a$ obeys \eqref{csda-7}, \eqref{csda-8},
we have

\begin{lemma}\label{csda-le1}
For all $E\in I$
the domain $D(\widetilde Q_{C,m}(E))$ of $\widetilde Q_{C,m}(E)$ satisfies
\bea
D(\widetilde Q_{C,m}(E))&=\{\varphi\in H_0^{(m,0)}(G\times S,\Gamma_-')\ |\ \textrm{there\ exists\ a\ sequence $\{\varphi_n\}\subset \cap_{k=1}^\infty C_0^{(k,0)}(G\times S,\Gamma_-')$} \nonumber \\
& \phantom{=\{}
{\rm and}\ h\in H_0^{(m,0)}(G\times S,\Gamma_-')\ {\rm such\ that}\ 
\n{\varphi_n-\varphi}_{H_0^{(m,0)}(G\times S,\Gamma_-')}\to 0\nonumber \\
& \phantom{=\{}
{\rm and}\ \n{\omega\cdot\nabla_x\varphi_n-h}_{H_0^{(m,0)}(G\times S,\Gamma_-')}\to 0\ {\rm as}\ n\to\infty\}
\eea
and therefore $D(\widetilde Q_{C,m}(E))$ is independent of $E$.
\end{lemma}

Let
\begin{align}\label{eq:rm_f:1}
{\rm f}(E)(x,\omega):=-{1\over{\hat a(x,E)}}e^{CE}\hat f(x,\omega,E).
\end{align}
Consider the evolution equation
\be\label{ecsd6}
{\p {\phi}E}-\widetilde Q_{C,m}(E)\phi={\rm f}(E),\quad \phi(0)=0,
\ee
where we interpret $\phi$  as a mapping $I\to L^2(G\times S)$ by defining $\phi(E)(x,\omega):=\phi(x,\omega,E)$.
The concept of solution of (\ref{ecsd6}) includes that $\phi(E)\in D(\widetilde Q_{C,m}(E))\subset H_0^{(m,0)}(G\times S,\Gamma_-)$ for any $E\in I$
and then that the solution of (\ref{ecsd6}) satisfies (strongly) the initial inflow boundary value problem
(\ref{se1a}), (\ref{se2a}),  (\ref{se3a}).

\begin{lemma}\label{csdale1-a}
Suppose  that  
\be\label{nis-15-a}
a\in  W^{\infty,(1,0)}(G\times I),
\ee
\be\label{nis-16-a}
-a(x,E)\geq\kappa>0\quad {\rm a.e.}
\ee
Then 
\begin{multline}\label{ac}  
\la -{1\over {\hat a(\cdot,E)}}\omega\cdot\nabla_x\varphi,\varphi\ra_{L^2(G\times S)}
\geq 
-C_0'\n{\varphi}_{L^2(G\times S)}^2
\\
\forall\varphi\in \widetilde W_{-,0}^2(G\times S):= \{\varphi\in\widetilde W^2(G\times S)\ |\ \varphi_{|\Gamma_-'}=0\},
\end{multline}
where $C_0':={1\over{2\kappa^2}}\n{\nabla_x a}_{L^\infty(G\times I)}$
and the space $\widetilde W^2(G\times S)$ is defined similarly to $\widetilde W^2(G\times S\times I)$ (see \eqref{fs12}).
\end{lemma}

\begin{proof}
See \cite[Lemma 3.4]{tervo17}.
\end{proof}

\begin{corollary}\label{csdale1-cor}
Suppose that (\ref{nis-15-a}) and (\ref{nis-16-a}) hold. 
Then for $\alpha\in\N_0^3$
\be\label{ac-p}  
\la -{1\over {\hat a(\cdot,E)}}\omega\cdot\nabla_x(\partial_x^\alpha\varphi),\partial_{x}^\alpha\varphi\ra_{L^2(G\times S)}
\geq 
-C_0'\n{\partial_{x}^\alpha\varphi}_{L^2(G\times S)}^2\quad \forall\varphi\in \cap_{k=1}^\infty C_0^{(k,0)}(\ol G\times S,\Gamma_-').
\ee
  
\end{corollary}

\begin{proof}

Noting that $\partial_{x}^\alpha\varphi\in\widetilde W_{-,0}^2(G\times S)$ for $\varphi\in \cap_{k=1}^\infty C_0^{(k,0)}(\ol G\times S,\Gamma_-')$ the assertion (\ref{ac}) follows from the previous Lemma \ref{csdale1-a}.

\end{proof}

Techniques analogous to the ones we have used above then yield

\begin{lemma}\label{csdale1}
Suppose  that  
\bea
&
a\in \s W^{\infty,(m,0)}(G\times I)\label{nis-15}
\\
&
-a(x,E)\geq\kappa>0,\ {\rm a.e.}\label{nis-16}
\\
&
\Sigma\in \s W^{\infty,(m,0,0)}(G\times S\times I)\label{nis-17}
\\
&
\sigma^2\in \s W_0^{\infty,(m,0,0)}(G\times S\times I,L^1(S'),\Gamma_-).\label{nis-18}
\eea
Then there exists a constant $C_0$ such that for any $C>C_0$ 
and any fixed $E\in I$
the operator $\widetilde Q_{C,m}(E):H_0^{(m,0)}(G\times S,\Gamma_-')\to H_0^{(m,0)}(G\times S,\Gamma_-')$ 
is $m$-dissipative.
\end{lemma}

\begin{proof}

Let $S_m(E)$ (for any fixed $E\in I$) be a linear operator defined by
\[
\begin{cases}
& D(S_m(E)):=\cap_{k=1}^\infty C_0^{(k,0)}(G\times S,\Gamma_-'),\\
& S_m(E)\varphi:=-{1\over {\hat a(\cdot,E)}}\omega\cdot\nabla_x\varphi.
\end{cases}
.
\]
Similarly to the proof of Lemma \ref{nis-le1}, we will show that
there exists a constant $C_0''$ such that for any $C'\geq C_0''$
\be\label{nis-18-a}
\la (\widetilde S_m(E)+C'\s I)\varphi,\varphi\ra_{H_0^{(m,0)}(G\times S,\Gamma_-')}\geq 0
\quad
\forall\varphi\in D(\widetilde S_m(E)),
\ee
where $\widetilde S_{m}(E):H_0^{(m,0)}(G\times S,\Gamma_-')\to H_0^{(m,0)}(G\times S,\Gamma_-')$ is the closure of $S_{m}(E)$ .

It suffices to show that 
\be\label{nis-18-a-a}
\la (S_m(E)+C'\s I)\varphi,\varphi\ra_{H_0^{(m,0)}(G\times S,\Gamma_-')}\geq 0
\quad
\forall\varphi\in \cap_{k=1}^\infty C_0^{(k,0)}(G\times S,\Gamma_-').
\ee
Using the formula
\bea\label{m1-1-o}
\la {1\over {\hat a(\cdot,E)}}\omega\cdot\nabla_x\varphi,\varphi\ra_{H_0^{(m,0)}(G\times S,\Gamma_-')}
={}&
\la {1\over {\hat a(\cdot,E)}}\omega\cdot\nabla_x\varphi,\varphi\ra_{H_0^{(m-1,0)}(G\times S,\Gamma_-')}
\nonumber\\
{}&
+
\sum_{|\alpha|=m}\la\partial_x^\alpha\big({1\over{\hat a(\cdot,E)}}\omega\cdot\nabla_x\varphi\big),\partial_x^\alpha\varphi\ra_{L^2(G\times S)}
\eea
and the Leibniz's rule
\[
&
\partial_x^\alpha\big({1\over{\hat a(\cdot,E)}}\omega\cdot\nabla_x\varphi\big)=
\sum_{\beta\leq \alpha}{\alpha \choose\beta}\partial_x^{\alpha-\beta}\big({1\over{\hat a(\cdot,E)}}\big)
\partial_x^\beta\big(\omega\cdot\nabla_x\varphi\big)\\
&
=
{1\over{\hat a(\cdot,E)}}\partial_x^\alpha\big(\omega\cdot\nabla_x\varphi\big)
+
\sum_{\beta< \alpha}{\alpha \choose\beta}\partial_x^{\alpha-\beta}\big({1\over{\hat a(\cdot,E)}}\big)
\partial_x^\beta\big(\omega\cdot\nabla_x\varphi\big)
\\
&
=
{1\over{\hat a(\cdot,E)}}\big(\omega\cdot\nabla_x(\partial_x^\alpha\varphi)\big)
+
\sum_{\beta< \alpha}{\alpha \choose\beta}\partial_x^{\alpha-\beta}\big({1\over{\hat a(\cdot,E)}}\big)
\big(\omega\cdot\nabla_x (\partial_x^\beta\varphi)\big)
\]
we get 
\[
&
\la {1\over {\hat a(\cdot,E)}}\omega\cdot\nabla_x\varphi,\varphi\ra_{H_0^{(m,0)}(G\times S,\Gamma_-')}
=
\la {1\over {\hat a(\cdot,E)}}\omega\cdot\nabla_x\varphi,\varphi\ra_{H_0^{(m-1,0)}(G\times S,\Gamma_-')}
\nonumber\\
&
+\sum_{|\alpha|= m}\sum_{\beta< \alpha}{\alpha \choose\beta}
\la \partial_x^{\alpha-\beta}\big({1\over{\hat a(\cdot,E)}}\big)
\big(\omega\cdot\nabla_x (\partial_x^\beta\varphi)\big),\partial_x^\alpha\varphi\ra_{L^2(G\times S)}
\nonumber\\
&
+\sum_{|\alpha|= m}\la {1\over {\hat a(\cdot,E)}}\omega\cdot\nabla_x(\partial_x^\alpha\varphi),\partial_x^\alpha\varphi\ra_{L^2(G\times S}
=:
R(\varphi)+
\sum_{|\alpha|= m}\la {1\over {\hat a(\cdot,E)}}\omega\cdot\nabla_x(\partial_x^\alpha\varphi),\partial_x^\alpha\varphi\ra_{L^2(G\times S)}.
\]

The term $R(\varphi)$ satisfies
\[
|R(\varphi)|\leq C_1\n{\varphi}_{H_0^{(m,0)}(G\times S,\Gamma_-')}^2
\]
whereas by Corollary \ref{csdale1-cor} the term
$\sum_{|\alpha|= m}\la -{1\over {\hat a(\cdot,E)}}\omega\cdot\nabla_x(\partial_x^\alpha\varphi),\partial_x^\alpha\varphi\ra_{L^2(G\times S)}$ 
obeys
\[
&
\sum_{|\alpha|= m}\la -{1\over {\hat a(\cdot,E)}}\omega\cdot\nabla_x(\partial_x^\alpha\varphi),\partial_x^\alpha\varphi\ra_{L^2(G\times S)} 
\geq 
-C_0'\sum_{|\alpha|= m}\n{\partial_{x}^\alpha\varphi}_{L^2(G\times S)}^2.
\]
Consequently, we have
\[
&
\la  S_m(E)\varphi,\varphi\ra_{H_0^{(m,0)}(G\times S,\Gamma_-')}
=
\la -{1\over {\hat a(\cdot,E)}}\omega\cdot\nabla_x\varphi,\varphi\ra_{H_0^{(m,0)}(G\times S,\Gamma_-')}
\nonumber\\
&
\geq 
-C_1\n{\varphi}_{H^{(m,0)}(G\times S,\Gamma_-')}^2-C_0'\sum_{|\alpha|=m}\n{\partial_{x}^\alpha\varphi}_{L^2(G\times S)}^2
\geq -(C_1+C_0')\n{\varphi}_{H^{(m,0)}(G\times S,\Gamma_-')}^2
\]
and so the assertion \eqref{nis-18-a-a} is valid with $C_0'':=C_1+C_0'$.

In other respects, the stated $m$-dissipativity of $\widetilde Q_{C,m}(E)$
follows by similar arguments as used in Theorems \ref{nis-th1} and \ref{nis-th2} above. 
We omit further details.
\end{proof}

To apply Theorem \ref{evoth} we need some additional assumptions.
We shall assume that
\be\label{nis-19}
{\sigma^2}\in C^1(I,W^{\infty,(m,0)}(G\times S,L^1(S')))\cap
C^1(I,W^{\infty,(m,0)}(G\times S',L^1(S))).
\ee

\begin{lemma}\label{le:K_C1}
Under the assumption \eqref{nis-19}, for any fixed $\varphi\in H^{(m,0)}(G\times S)$
the map
\[
k_\varphi:I\to H^{(m,0)}(G\times S);\quad k_\varphi(E)=\widehat{K}_r(E)\varphi
\]
is in $C^1(I, H^{(m,0)}(G\times S))$ and its derivative is
$\frac{dk_\varphi}{dE}(E)=\p{\widehat K_r}{E}(E)\varphi$,
where for any fixed $E$ the operator ${\p {\widehat K_r}{E}}(E): H^{(m,0)}(G\times S)\to H^{(m,0)}(G\times S)$ is
defined by
\be\label{K1}
\Big({\p {\widehat K_r}{E}}(E)\varphi\Big)(x,\omega):=\int_{S'}{\p {\hat\sigma^2}{E}}(x,\omega',\omega,E)\varphi(x,\omega') d\omega'.
\ee
\end{lemma}

\begin{proof}
We have
\[
{{k_\varphi(E+h)-k_\varphi(E)}\over h}(x,\omega)=\int_{S'} {{\hat\sigma^2(x,\omega',\omega,E+h)-\hat\sigma^2(x,\omega',\omega,E)}\over h}\varphi(x,\omega') d\omega'
\]
and for its derivatives
\[
&
\partial_x^\alpha\big({{k_\varphi(E+h)-k_\varphi(E)}\over h}\big)(x,\omega)\\
&
=\sum_{\beta\leq \alpha}{\alpha \choose\beta}\int_{S'} {{\partial_x^{\alpha-\beta}\big(\hat\sigma^2(x,\omega',\omega,E+h)-\hat\sigma^2(x,\omega',\omega,E)\big)}\over h}(\partial_x^\beta\varphi)(x,\omega') d\omega',\ |\alpha|\leq m.
\]
In addition,
\[
\partial_x^\alpha\Big({\p {\widehat K_r}{E}}(E)\varphi\Big)(x,\omega)=\sum_{\beta\leq \alpha}{\alpha \choose\beta}\int_{S'}\partial_x^{\alpha-\beta}\big({\p {\hat\sigma^2}{E}}\big)(x,\omega',\omega,E)(\partial_x^\beta\varphi)(x,\omega') d\omega'.
\]
Hence for $ |\alpha|\leq m$
\bea\label{nis-20}
&
\partial_x^\alpha\Big({{k_\varphi(E+h)-k_\varphi(E)}\over h}
-{\p {\widehat K_r}{E}}(E)\varphi\Big)(x,\omega)\\
&
=
\sum_{\beta\leq \alpha}{\alpha \choose\beta}\int_{S'}\Big( \partial_x^{\alpha-\beta}\big({{\hat\sigma^2(x,\omega',\omega,E+h)-\hat\sigma^2(x,\omega',\omega,E)}\over h}-
{\p {\hat \sigma^2}{E}}\big)(x,\omega',\omega,E)\Big)
(\partial_x^\beta\varphi)(x,\omega') d\omega'.
\eea
This implies due to estimates like \eqref{nis-10a} and \eqref{nis-10} that
\bea\label{nis-21}
&
\n{\partial_x^\alpha\Big({{k_\varphi(E+h)-k_\varphi(E)}\over h}
-{\p {\widehat K_r}{E}}(E)\varphi\Big)}_{L^2(G\times S)}^2\nonumber\\
&
\leq 
C_\alpha
\sum_{\beta\leq \alpha}{\alpha \choose\beta}^2
\n{ \partial_x^{\alpha-\beta}\big({{\hat\sigma^2(\cdot,\cdot,\cdot,E+h)-\hat\sigma^2(\cdot,\cdot,\cdot,E)}\over h}-
{\p {\hat \sigma^2}{E}}(\cdot,\cdot,\cdot,E)\big)}_{L^\infty(G\times S,L^1(S'))}\nonumber\\
&
\cdot
\n{\partial_x^{\alpha-\beta}\big({{\hat\sigma^2(\cdot,\cdot,\cdot,E+h)-\hat\sigma^2(\cdot,\cdot,\cdot,E)}\over h}-
{\p {\hat \sigma^2}{E}}(\cdot,\cdot,\cdot,E)\big)}_{L^\infty(G\times S',L^1(S))}
\n{\partial_x^\beta\varphi}_{L^2(G\times S)}^2 
\eea
and so we find by the assumption (\ref{nis-19}) that 
\[
{{k_\varphi(E+h)-k_\varphi(E)}\over h}\longto_{h\to 0} \p{\widehat K_r}{E}(E)\varphi\quad {\rm in}\  H^{(m,0)}(G\times I).
\]
Thus the derivative of $k_\varphi$ is  ${{dk_\varphi}\over {dE}}(E)={\p {\widehat K_r}{E}}(E)\varphi$ for every $E\in I$.
In addition, by the assumption \eqref{nis-19} and by an estimate like \eqref{nis-21} we see that ${{dk_\varphi}\over {dE}}$ is continuous $I\to H^{(m,0)}(G\times S)$ which completes the proof.

\end{proof}

With the above notations and results at hand, we are ready to state the next regularity result.

\begin{theorem}\label{evoth1}
Suppose that   
\bea
&
\Sigma\in C^1(I,\s W^{\infty,(m,0)}(G\times S)),\label{ecsd6a-b} \\
& 
a\in C^1(I,\s W^{\infty,m}(G)),\label{ecsd6a-a-b}\\
&
-a(x,E)\geq \kappa>0\quad {\rm a.e.},\label{nis-16-b}\\
&
\sigma^2\in \s W_0^{\infty,(m,0,0)}(G\times S\times I,L^1(S'),\Gamma_-)\nonumber\\
&
\quad\quad \cap 
C^1(I,W^{\infty,(m,0)}(G\times S,L^1(S')))\cap
C^1(I,W^{\infty,(m,0)}(G\times S',L^1(S))).
\eea
Then for any $f\in C^1(I, H_0^{(m,0)}(G\times S,\Gamma_-'))$ the problem
\be\label{nis-22}
a{\p \psi{E}}
+\omega\cdot\nabla_x\psi+\Sigma\psi-K_{r}\psi=f,\quad \psi_{|\Gamma_-}=0,\quad \psi(\cdot,\cdot,E_m)=0
\ee 
has a unique (strong) solution
$\psi\in C^1(I,H_0^{(m,0)}(G\times S,\Gamma_-'))$.
\end{theorem}

\begin{proof}
Let $f\in C^1(I, H_0^{(m,0)}(G\times S,\Gamma_-'))$ and define ${\rm f}(E)$ as in \eqref{eq:rm_f:1}.
The assumptions made on $a$ imply that ${\rm f}\in C^1(I,H_0^{(m,0)}(G\times S,\Gamma_-'))$.

Choose a constant $C_0$ as in Lemma \ref{csdale1}
and let $C>C_0$.
We make the change of variables and the change of unknown function as above by setting $\tilde\psi(x,\omega,E)=\psi(x,\omega,E_m-E)$ and $\phi=e^{CE}\tilde\psi$.
Consider the problem (see \eqref{ecsd6})
\be\label{ecsd7}
{\p {\phi}E}-\widetilde Q_{C,m}(E)\phi={\rm f}(E),\quad \phi(0)=0,
\ee
Due to Lemma \ref{csda-le1} the domain $D(\widetilde Q_{C,m}):=D(\widetilde Q_{C,m}(E))$  is independent of $E$.
Moreover by Lemma \ref{csdale1} the (densely defined and closed) operator
$\widetilde Q_{C,m}(E):H_0^{(m,0)}(G\times S,\Gamma_-')\to H_0^{(m,0)}(G\times S,\Gamma_-')$ is $m$-dissipative for any fixed $E\in I$. 

The assumptions imply that for any fixed $\varphi\in D(\widetilde Q_{C,m})$
the mapping
\[
h_{\varphi}:I\to H_0^{(m,0)}(G\times S,\Gamma_-');\quad h_\varphi(E):=\widetilde Q_{C,m}(E)\varphi,
\]
is differentiable and its derivative is
\[
{{dh_\varphi}\over {dE}}(E)=&{\partial\over{{\partial E}}}\Big({1\over{\hat a(E)}}\Big)\omega\cdot \nabla_x\varphi
+{\partial\over{{\partial E}}}\Big({1\over{\hat a(E)}}\widehat\Sigma(E)\Big)\varphi
 \\
&
-{\partial\over{{\partial E}}}\Big({1\over{\hat a(E)}}\Big)\widehat K_r(E)\varphi
-{1\over{\hat a(E)}}{\p {\widehat K_r}{E}}(E)\varphi,
\]
where ${\p {\widehat K_r}{E}}(E)\varphi$ is defined in \eqref{K1},
and the derivative $\pa{E}(\widehat{K}_r(E)\varphi)={\p {\widehat  K_r}{E}}(E)\varphi$ 
is provided by Lemma \ref{le:K_C1}.
By assumptions on $\Sigma,\ a$ and $\sigma^2$ we thus see that $h_\varphi$ is in $C^1(I,H_0^{(m,0)}(G\times S,\Gamma_-'))$.

By Theorem \ref{evoth}
there exists a unique solution $\phi\in C(I,D(\widetilde Q_{C,m}))\cap C^1(I,H_0^{(m,0)}(G\times S,\Gamma_-'))$ of (\ref{ecsd7}).
Then $\psi(x,\omega,E):=e^{-C(E_m-E)}\phi(E_m-E)(x,\omega)$ is the desired (strong) solution of 
the problem \eqref{nis-22}. This completes the proof.
\end{proof}

\begin{remark}\label{nis-re4}
Consider the problem with non-zero inflow data
\be\label{nis-22-e}
a{\p \psi{E}}
+\omega\cdot\nabla_x\psi+\Sigma\psi-K_{r}\psi=f,\quad \psi_{|\Gamma_-}=g,\quad \psi(\cdot,\cdot,E_m)=0,
\ee 
where $g\in C^2(I,T^2(\Gamma_-'))$ such that the compatibility condition 
\be\label{cc}
g(\cdot,\cdot,E_m)=0.
\ee
holds.
Then the lift $L_-g\in C^2(I,\widetilde W^2(G\times S))$ for which $(L_-g)_{|\Gamma_-}=g$, and $(L_-g)(\cdot,\cdot,E_m)=0$ (follows from (\ref{cc})),
and furthermore $\omega\cdot\nabla_x(L_-g)=0$. Substituting $u:=\psi-L_-g$ for $\psi$ we obtain the following problem for $u$,
\begin{gather}
a{\p {u}E}+\omega\cdot\nabla_x u+\Sigma u
- K_ru = h,\nonumber\\
u_{|\Gamma_-}=0,\nonumber\\
u(\cdot,\cdot,E_m)=0,\label{ecsd8}
\end{gather}
where
\[
h:=f-\Big(a{\p {(L_-g)}E}+\Sigma (L_-g)-K_r(L_-g)\Big).
\]
 
Supposing that $a,\Sigma,\sigma^2$ obey the assumptions of Theorem \ref{evoth1} and that
\[
h\in C^1(I, H_0^{(m,0)}(G\times S,\Gamma_-'))
\]
we are able (similarly to Remark \ref{nis-co1}) to seek sufficient conditions under which
\[
u\in C^1(I,H_0^{(m,0)}(G\times S,\Gamma_-'))
\]
which implies that
\[
\psi\in C^1(I,H_0^{(m,0)}(G\times S,\Gamma_-'))+L_-g.
\]
\end{remark}

\begin{example}\label{csda-ex1}

Consider the  problem 
\[
-{\p {\psi}E}+\omega\cdot\nabla_x\psi+
\Sigma\psi
= f,\ \
\psi_{|\Gamma_-}=0,\quad 
\psi(\cdot,\cdot,E_m)=0,
\]
where $\Sigma$ is constant.
The solution is (\cite[section 10]{tervo18-up}, \cite[pp. 226-235]{dautraylionsv6})
\[
\psi(x,\omega,E)=\int_0^{\min\{E_m-E,\, t(x,\omega)\}} e^{-\Sigma s} f(x-s\omega,\omega,E+s)ds
\]
As in Lemma \ref{nis-le2} above, we see that for
$f\in C(\ol G\times S\times I)$ such that ${\rm supp}(f)\cap (\Gamma_- \cup\Gamma_0)=\emptyset$ the solution $\psi$ satisfies
${\rm supp}(\psi)\cap (\Gamma_- \cup\Gamma_0)=\emptyset$. Moreover, the solution $\psi$ can be seen to belong to
$C(I,H_0^{(m,0)}(G\times S,\Gamma_-'))$
by explicit computations for this kind of $f$ (see the proof of Theorem \ref{nis-th1}, Parts A.1-A.2.).
\end{example}

\vskip3mm

\begin{remark}\label{comp-re}

We end this section with the following remarks concerning for compatibility conditions and generalizations of the above results.
The initial boundary value problems must satisfy certain compatibility conditions in order to obey higher order regularity. Suppose that $K_r=K_r^2$ (see \eqref{def:K_r^2:1}).
Let us analyse the problem
\[
&
a(x,E){\p \psi{E}}
+\omega\cdot\nabla_x\psi+\Sigma\psi-K_{r}\psi=f,\\
&
{\psi}_{|\Gamma_-}=g,\ 
\psi(\cdot,\cdot,E_m)=0
\]
or equivalently
\bea
&
{\p \psi{E}}
-P(x,\omega,E,D)\psi=F \ {\rm in}\ G\times S\times I^\circ, \label{gen-10a}
\\
&
{\psi}=g \ {\rm on }\ \Gamma_-,\label{gen-10b}
\\
&
\psi(\cdot,\cdot,E_m)=0\ {\rm in}\ G\times S, \label{gen-10c}
\eea
where
\[
&
P(x,\omega,E,D)\psi
:=-{1\over{a(x,E)}}\big( \omega\cdot\nabla_x\psi+\Sigma(x,\omega,E)\psi-K_r\psi\big),\\
& 
F:={1\over{a(x,E)}}f.
\]

Suppose that $\psi\in H^{(m_1,m_2,m_3)}(G\times S\times I)$ for large enough $m_j,\ j=1,2,3$
(so that the formal computations given below are legitimate up to the boundary). 
The initial and boundary conditions (\ref{gen-10c}), (\ref{gen-10b}) imply
\be\label{gen-11}
g(\cdot,\cdot,E_m)=0\ {\rm on}\ \Gamma'_-
\ee
which is a \emph{zeroth order compatibility condition}.

Furthermore, by (\ref{gen-10b}), (\ref{gen-10a})
\[
{\p g{E}}(\cdot,\cdot,E_m)={\p \psi{E}}(\cdot,\cdot,E_m)
=\big(P(x,\omega,E,D)\psi\big)(\cdot,\cdot,E_m)+F(\cdot,\cdot,E_m)\  {\rm on}\ \Gamma'_-,
\]
where by (\ref{gen-10c})
\[
&
\big(P(x,\omega,E,D)\psi\big)(\cdot,\cdot,E_m)\\
&
=-{1\over{a(x,E_m)}}\big(
\omega\cdot\nabla_x\psi(\cdot,\cdot,E_m)+\Sigma(\cdot,\cdot,E_m)\psi(\cdot,\cdot,E_m)-(K_r\psi)(\cdot,\cdot,E_m)\big)
=0
\]
on $G\times S$ and hence on $\Gamma'_-$.
Therefore
\[
{\p g{E}}(\cdot,\cdot,E_m)=F(\cdot,\cdot,E_m)\ {\rm on}\ \Gamma'_-
\]
which is a \emph{first order compatibility condition}.

By similar computations,
we find a \emph{second order compatibility condition} 
\[
{\q g{E}}(\cdot,\cdot,E_m)=\big(P(x,\omega,E,D)F\big)(\cdot,\cdot,E_m)+{\p F{E}}(\cdot,\cdot,E_m)\ {\rm on}\ \Gamma'_-.
\]
In analogous way we are able to compute higher order compatibility conditions
(cf. \cite{audiard23}, \cite[Introduction]{rauch74}). These conditions must be fulfilled
up to order $m_3-1$ when the solution has $E$-regularity up to order $m_3$.

\begin{example}
For example, consider the following case.
Let $m=(m_1,m_2,m_3)\in \N_0^3$ and let
(compare with \eqref{eq:def:C^k_ol_G_Gamma_-:1})
\[
C_0^{k}(\ol G\times S\times I,\Gamma_-\cup\{E_m\})
\!:=\!{}& \{ \psi\in C^k(\ol G\times S\times I)\, |\, {\rm supp}(\psi)\cap \big(\Gamma_- \cup\Gamma_0\cup (\ol G\times S\times \{E_m\})\big)=\emptyset\}.
\]
Let $ H_0^{(m_1,m_2,m_3)}(G\times S\times I,\Gamma_-\cup\{E_m\})$ be the completion of 
$\cap_{k=1}^\infty C_0^k(\ol G\times S\times I,\Gamma_-\cup\{E_m\})$ with respect
to the inner product $\la \psi,v\ra_{ H^{(m_1,m_2,m_3)}(G\times S\times I)}$
(see \eqref{hminner}).
Consider the  problem 
\[
-{\p {\psi}E}+\omega\cdot\nabla_x\psi+
\Sigma\psi
= f,\ \
\psi_{|\Gamma_-}=0,\quad 
\psi(\cdot,\cdot,E_m)=0,
\]
where $\Sigma$ is constant.
As mentioned above in Example \ref{csda-ex1}, the solution is 
\be\label{gen-1}
\psi(x,\omega,E)=\int_0^{\min\{E_m-E,\, t(x,\omega)\}} e^{-\Sigma s} f(x-s\omega,\omega,E+s)ds.
\ee
By direct computations
one finds that for $f\in  H_0^{(m_1+m_2+m_3,m_2,m_3)}(G\times S\times I,\Gamma_-\cup\{E_m\}),\ m_3\geq 1$,
the unique solution $\psi$ is in $\cap_{j=0}^{m_2}\cap_{k=0}^{m_3} H_0^{(m_1+m_2+m_3-j-k,j,k)}(G\times S\times I,\Gamma_-\cup\{E_m\})$.
We note that for $g=0$ and $f\in H_0^{(m_1+m_2+m_3,m_2,m_3)}(G\times S\times I,\Gamma_-\cup\{E_m\})$
the needed compatibility conditions hold up to order $m_3-1$.
In particular, in the situation of Theorem \ref{evoth1} we have $m_3=1$ and the zeroth order compatibility condition holds.
\end{example}

It is expected that 
in spaces $ H_0^{(m_1,m_2,m_3)}(G\times S\times I,\Gamma_-\cup\{E_m\})$ 
a regularity result along the lines of the above example can be proved for more general types of transport equations.
By using appropriate lifts we additionally conjecture that the spaces $H_0^{(m_1,m_2,m_3)}(G\times S\times I,\Gamma_-\cup\{E_m\})$ can be replaced by the larger spaces 
$H_0^{(m_1,m_2,m_3)}(G\times S\times I,\Gamma_-)$.
\end{remark} 

\vskip1cm

\section{Appendix}\label{somele}

We recall the following result concerning the change of derivation and integration

\begin{theorem}\label{ointder}
Let $G\subset\R^n$ be an open set and let $(Y,\mc{F},\mu)$ be a measure space.
Suppose that $f:G\times Y\to \R$ is $\mc{B}(G)\otimes\mc{F}$-measurable mapping such that
the partial derivatives ${{\partial^\alpha f}\over{\partial x^\alpha}}$ exist on $G\times Y$ for $|\alpha|\leq r$ and that there exist $g_\alpha\in L^1(Y,\mc{F},\mu)$ such that
\[
\Big|{{\partial^\alpha f}\over{\partial x^\alpha}}(x,y)\Big|
\leq g_\alpha(y)\quad \forall x\in G,\ y\in Y,\ |\alpha|\leq r.
\]
Then
\be\label{orgs7a}
{{\partial^\alpha }\over{\partial x^\alpha}}\Big(\int_Y f(x,y) \mu(dy)\Big)
=\int_Y{{\partial^\alpha f}\over{\partial x^\alpha}}(x,y) \mu(dy)
\quad \forall x\in G,\ |\alpha|\leq r.
\ee
\end{theorem}

\begin{proof}
For the proof we refer e.g. to \cite{folland}, p. 56.
\end{proof}

\medskip
\subsection{Proof of Lemma \ref{le:basic_geometric_properties:1}}\label{app:le:basic_geometric_properties:1:proof}

\begin{proof}
From the assumptions imposed on the domain $G$,
we deduce that $G$ is open, bounded and convex (since its closure $\ol{G}$ is strictly convex)
and that $\ol{G}$ is a $C^1$-submanifold of $\R^3$ with boundary.
The proof given below only uses these properties of $G$.

\medskip
\noindent
{\bf (i)}:
Pick an open ball neighbourhood $V=B(x,\delta)$, $\delta>0$, of $x$
and a $C^1(V)$-function $r$
such that $G\cap V=\{r<0\}$.
Then $(\partial G)\cap V=\{r=0\}$
and $\nu(x)=\nabla r(x)/\n{\nabla r(x)}$ for all $x\in (\partial G)\cap V$.

For all $0\leq s<\delta$ it holds $x-s\omega\in V$
and $h(s):=\dif{s} r(x-s\omega)=-\omega\cdot (\nabla r)(x-s\omega)$.
Since $h(0)=-\omega\cdot (\nabla r)(x)$
and $\omega\cdot\nu(x)>0$ (since $(x,\omega)\in\Gamma'_+$),
we have $h(0)<0$.
Therefore, there is $0<s_0\leq \delta$ such that
$h(s)<0$ for all $s\in [0,s_0]$.
As $r(x)=0$ since $x\in\partial G$,
it follows that
\[
r(x-s\omega)=r(x-s\omega)-r(x)=\int_0^s h(t)dt<0
\]
for all $0<s\leq s_0$,
that is,
\[
x-s\omega\in \{r<0\}=G\cap V\subset G\quad \forall s\in [0,s_0].
\]

This shows that there exists $s_0>0$ such that
$x-s\omega\in G$ for all $0<s\leq s_0$.
From the definition of $\tau_+$ in \eqref{def:tau_pm:1},
we deduce that $\tau_+(x,\omega)\geq s_0>0$.
This completes the proof of \emph{(i)}.

\medskip
\noindent
{\bf (ii)}:
This follows by the same arguments as \emph{(i)}.

\medskip
\noindent
{\bf (iii)}:
Like in case \emph{(i)},
pick an open ball neighbourhood $B(x,\delta)$, $\delta>0$, of $x$
and a $C^1(V)$-function $r$ such that
$G\cap V=\{r<0\}$,
$(\partial G)\cap V=\{r=0\}$
and $\nu(x)=\nabla r(x)/\n{\nabla r(x)}$ for all $x\in (\partial G)\cap V$.

Let $z\in G$. Since $x,z\in \ol{G}$ and $\ol{G}$ is convex,
we have $(1-t)x+tz=x+t(z-x)\in \ol{G}$ for all $t\in [0,1]$
and hence $r(x+t(z-x))\leq 0$ for all $t\in [0,1]$.
Since $x\in\partial G$, we have $r(x)=0$
and therefore
\[
\frac{r(x+t(z-x)) - r(x)}{t}=\frac{r(x+t(z-x))}{t}\leq 0
\]
for all $0<t\leq 1$. In the limit $t\to 0+$ this implies
that
\[
(\nabla r)(x)\cdot (z-x)=\dif{t}\big|_{t=0} r(x+t(z-x))\leq 0.
\]
We have thus shown that
$(\nabla r)(x)\cdot (z-x)\leq 0$ for all $z\in G$
and hence
\begin{align}\label{eq:le:basic_geometric_properties:1:1}
\nu(x)\cdot (z-x)\leq 0\quad \forall z\in G.
\end{align}

Suppose now that there exists $z_0\in G$
for which $\nu(x)\cdot (z_0-x)=0$.
Since $G$ is open, we have $B(z_0,\delta')\subset G$ for some small $\delta'>0$
and hence by \eqref{eq:le:basic_geometric_properties:1:1},
\[
\nu(x)\cdot (z_0+h-x)\leq 0\quad \forall \n{h}<\delta',
\]
i.e.,
since $\nu(x)\cdot (z_0-x)=0$,
\[
\nu(x)\cdot h\leq 0\quad \forall \n{h}<\delta',
\]
which is clearly absurd
(for instance, if we take $h=\frac{\delta'}{2}\nu(x)$,
we have $\n{h}=\frac{\delta'}{2}<\delta'$
and hence on one hand
$\nu(x)\cdot h\leq 0$,
while on the other hand
$\nu(x)\cdot h=\frac{\delta'}{2}\nu(x)\cdot \nu(x)>0$).
This contradiction proves that there is no $z_0\in G$
for which $\nu(x)\cdot (z_0-x)=0$.

From this and \eqref{eq:le:basic_geometric_properties:1:1}
we deduce that $\nu(x)\cdot (z-x)<0$ for all $z\in G$.
This completes the proof of \emph{(iii)}.

\medskip
\noindent
{\bf (iv)}:
If $(x,\omega)\in G\times S$,
then by the fact that $G$ is open convex and bounded,
there exists $s_0>0$ such that
$x-s\omega\in G$ for all $0<s<s_0$
and $y:=x-s_0\omega\in\partial G$ (we omit details).
On the other hand,
if $(x,\omega)\in\Gamma'_+$
then it follows from \emph{(i)}
that there exists $s'_0>0$ such that
$x-s\omega\in G$ for all $0<s\leq s_0'$.
Letting $x':=x-s_0'\omega\in G$,
we have $(x',\omega)\in G\times S$
and thus by the above,
there exists $s''_0>0$ such that
$x'-s\omega\in G$ for all $0\leq s<s''_0$
and $y:=x'-s_0''\omega\in\partial G$.
It follows that
$y=(x-s_0'\omega)-s_0''\omega=x-s_0\omega$
with $s_0:=s_0'+s_0''>0$
and $x-s\omega\in G$ for all $0<s<s_0$.

Thus so far, we have shown the following:
Regardless whether $(x,\omega)\in G\times S$ or $(x,\omega)\in\Gamma'_+$,
there exists $s_0>0$
such that $x-s\omega\in G$ for all $0<s<s_0$
and $y:=x-s_0\omega\in\partial G$.

Since $x=y+s_0\omega$ and thus $y+(s_0-s)\omega=x-s\omega\in G$ for all $0<s<s_0$,
it follows from case \emph{(iii)} that $\nu(y)\cdot ((s_0-s)\omega)<0$
for all $0<s<s_0$.
This implies that
$\omega\cdot\nu(y)<0$
i.e. $(y,\omega)\in\Gamma'_-$.

Finally, as $x-s\omega\in G$ for $0<s<s_0$ and $x-s_0\omega=y\in\partial G$,
we deduce that $\tilde{t}(x,\omega)=s_0>0$
and hence $x=y+\tilde{t}(x,\omega)\omega$.
The proof of \emph{(iv)} is therefore complete.

\end{proof}

\medskip
\subsection{Proof of Proposition \ref{pr:continuity_of_tilde_t:1}}\label{app:pr:continuity_of_tilde_t:1:proof}

\begin{proof}
Let $(x_0,\omega_0)\in G\times S$ and write $y_0:=x_0-t(x_0,\omega_0)\omega_0$.
According to (iv) of Lemma \ref{le:basic_geometric_properties:1}, we have $(y_0,\omega_0)\in\Gamma'_-$,
that is, $\omega_0\cdot\nu(y_0)<0$
and therefore by Proposition 4.7 in \cite{tervo17-up} the escape time map $t:G\times S\to\R_+$
is continuously differentiable on a neighbourhood of $(x_0,\omega_0)$.
This proves that $t\in C^1(G\times S)$.
In particular, $t$ is continuous on $G\times S$
(see also Lemma 2.5 in \cite{tervo18-up}).

Since $\tilde{t}=0$ on $\Gamma'_-\cup\Gamma'_0$ (see \eqref{olt}),
it follows from (ii) of Lemma \ref{le:tilde_t_bounded_from_below:1} (see below)
that $\tilde{t}$ is continuous at points of $\Gamma'_-\cup\Gamma'_0$.
Because also $\tilde{t}|_{G\times S}=t$ is continuous at points of $G\times S$,
we conclude that
$\tilde{t}$ is continuous on $(G\times S)\cup\Gamma'_-\cup\Gamma'_0$.

Let $(x_0,\omega_0)\in \Gamma'_+$ for the remainder of the proof.
By (iv) of Lemma \ref{le:basic_geometric_properties:1},
it holds $(x_0-\tau_+(x_0,\omega_0)\omega_0,\omega_0)=(x_0-\tilde{t}(x_0,\omega_0)\omega_0,\omega_0)\in\Gamma'_-$.
Hence by Lemma 2.5 in \cite{tervo18-up}, the map $t$ has the limit
\be\label{aa-1}
\lim_{G\times S\ni (x,\omega)\to (x_0,\omega_0)} t(x,\omega)=\tau_+(x_0,\omega_0)
\ee
that is $\tilde{t}(x,\omega)\to \tilde{t}(x_0,\omega_0)$
when $(x,\omega)$ tends to $(x_0,\omega_0)$ from $G\times S$.

Thus, in order to conclude that $\tilde{t}$ is continuous on all of $\ol{G}\times S$,
it remains to show that $\tau_+$ is continuous on $\Gamma'_+$.

Notice that if $(y,\omega)\in \Gamma'_+$
then from (i) of Lemma \ref{le:basic_geometric_properties:1},
we have $\tau_+(y,\omega)>0$
and hence the definitions of $t$ and $\tau_+$
(see \eqref{def:t:1}, \eqref{def:tau_pm:1})
imply that $y-s\omega\in G$ and
$\tau_+(y,\omega)=s+t(y-s\omega,\omega)$
for all $0<s<\tau_+(y,\omega)$.

Let $(x_k,\omega_k)$, $k=1,2,\dots$, be a sequence in $\Gamma'_+$
converging to $(x_0,\omega_0)$.
If the sequence $\tilde{t}(x_k,\omega_k)$, $k=1,2,\dots$,
were \emph{not} bounded from below by a constant $c>0$,
then it would have a subsequence $\tilde{t}(x_{k_i},\omega_{k_i})$, $i=1,2,\dots$,
converging to $0$
and hence (ii) of Lemma \ref{le:tilde_t_bounded_from_below:1}
would imply that $(x_0,\omega_0)\in\Gamma'_-\cup\Gamma'_0$
which is absurd.

This shows that there must exist a constant $c>0$
such that $\tilde{t}(x_k,\omega_k)\geq c$ for all $k=1,2,\dots$
Since $\tau_+(x_0,\omega_0)>0$ by (i) of Lemma \ref{le:basic_geometric_properties:1}, we may assume that $0<c<\tau_+(x_0,\omega_0)$
Then for all $0<s<c$ and all $k$,
we have $x_k-s\omega_k\in G$ and
\[
\tau_+(x_k,\omega_k)=s+t(x_k-s\omega_k,\omega_k)
\longto_{k\to\infty} s+t(x_0-s\omega_0,\omega_0)
=\tau_+(x_0,\omega_0),
\]
where the limit exists and equals the right-hand side because of \eqref{aa-1}
and $(x_k,\omega_k)\to (x_0,\omega_0)$ as $k\to\infty$
(note that $0<s<\tau_+(x_0,\omega_0)$ by our choice of $c$
and hence $(x_0-s\omega_0,\omega_0)\in G$).

The above argument shows that
if $(x_0,\omega_0)\in \Gamma'_+$
and if $(x_k,\omega_k)$, $k=1,2,\dots$, be a sequence in $\Gamma'_+$
converging to $(x_0,\omega_0)$,
then $\tau_+(x_k,\omega_k)=\tilde{t}(x_k,\omega_k)\to \tilde{t}(x_0,\omega_0)=\tau_+(x_0,\omega_0)$
as $k\to\infty$.
In other words, $\tilde{t}$ is also continuous on points of $\Gamma'_+$.

As mentioned above, this completes the proof that $\tilde{t}$
is continuous on $\ol{G}\times S$.
\end{proof}

\begin{lemma}\label{le:tilde_t_bounded_from_below:1}
Suppose that $(y,\omega)\in\partial G\times S$
and $(x_k,\omega_k)$, $k=1,2,\dots$,
is a sequence in $\ol{G}\times S$ converging to $(y,\omega)$.
\begin{itemize}
\item[(i)]
If the sequence $\tilde{t}(x_k,\omega_k)$, $k=1,2,\dots$,
is bounded from below by some constant $c>0$
then $(y,\omega)\in\Gamma'_+$ and $(x_k,\omega_k)\in (G\times S)\cup\Gamma'_+$ for all $k=1,2,\dots$.
\item[(ii)] We have $(y,\omega)\in\Gamma'_-\cup\Gamma'_0$
if and only if $\tilde{t}(x_k,\omega_k)\to 0$ as $n\to\infty$.
\end{itemize}
\end{lemma}

\begin{proof}

\noindent {\bf (i):}
If $(x_0,\omega_0)\in \ol{G}\times S$
is such that $\tilde{t}(x_0,\omega_0)=0$,
it follows from the definition \eqref{olt} of $\tilde{t}$
that $(x_0,\omega_0)\in \Gamma'_-\cup\Gamma'_0$.
That is, if $(x_0,\omega_0)\in \ol{G}\times S$
and $\tilde{t}(x_0,\omega_0)>0$ then $(x_0,\omega_0)\in (G\times S)\cup\Gamma'_+$.

Let $c>0$ be such that $\tilde{t}(x_k,\omega_k)\geq c$ for all $k$.
Then by what we have just discussed, it holds
$(x_k,\omega_k)\in (G\times S)\cup\Gamma'_+$ for all $k$.

The rest of the proof of (i) is divided into three parts.

\medskip
\noindent {\bf Part I:}
If $0<s\leq c$, it then follows from the definition
of $t$ and $\tau_+$ (see \eqref{def:t:1} and \eqref{def:tau_pm:1})
that $x_k-s\omega_k\in G$ for all $k$.
In the limit $k\to\infty$, we have $y-s\omega\in\ol{G}$
for all $0<s\leq c$.

\medskip
\noindent {\bf Part II:}
We claim that there exists $s_1>0$
such that $y-s\omega\in G$ for all $0<s<s_1$.

First, if $y-s\omega\in G$ for all $0<s<c$,
we may choose $s_1:=c$.

On the other hand, if $y-s_0\omega\notin G$ for some $0<s_0<c$,
then because $y-s_0\omega\in\ol{G}$ by \emph{Part I},
it must hold $y-s_0\omega\in \partial G$.
Since also $y\in\partial G$ and because $\ol{G}$ is strictly convex,
so these points of $\partial G$ are its extreme points,
we must have $y-s\omega\in G$ for all $0<s<s_0$.
Hence we may choose $s_1:=s_0$.

This proves the claim.

\medskip
\noindent {\bf Part III:}
By \emph{Part II} there is $s_1>0$ such that
$y-s\omega\in G$ for all $0<s<s_1$.
Therefore, (iii) of Lemma \ref{le:basic_geometric_properties:1}
implies that $\nu(y)\cdot((y-s\omega)-y)<0$
i.e. $-s\omega\cdot \nu(y)<0$ for all $0<s<s_1$.
This means that $\omega\cdot \nu(y)>0$, i.e. $(y,\omega)\in\Gamma'_+$,
and hence the proof of (i) is complete.

\medskip
\noindent {\bf (ii):}
Write $\tilde{t}_k:=\tilde{t}(x_k,\omega_k)\geq 0$, $k=1,2,\dots$,

Let first $(y,\omega)\in\Gamma'_-\cup\Gamma'_0$
Let $\tilde{t}_{k_i}$, $i=1,2,\dots$, be any subsequence of $\tilde{t}_k$, $k=1,2,\dots$
Since $(x_{k_i},\omega_{k_i})\to (y,\omega)$ when $i\to\infty$,
it follows from (i) that $\tilde{t}_{k_i}$, $i=1,2,\dots$,
\emph{cannot} be bounded from below by a constant $c>0$
and therefore it has a further subsequence
$\tilde{t}_{k_{i_n}}$, $n=1,2,\dots$,
converging to $0$.
From this one concludes that the original sequence
$\tilde{t}_k$, $k=1,2,\dots$ converges to $0$.

Conversely, suppose that $(y,\omega)\in\partial G\times S$
and $\tilde{t}_k\to 0$ as $k\to\infty$.
Given $k=1,2,\dots$,
if $(x_k,\omega_k)\in\Gamma'_-\cup\Gamma'_0$
define $\ol{x}_k:=x_k$
so that $\ol{x}_k\in \in\Gamma'_-\cup\Gamma'_0$,
and if $(x_k,\omega_k)\in(G\times S)\cup \Gamma'_+$,
define $\ol{x}_k:=x_k-\tilde{t}_k\omega_k$
so that $(\ol{x}_k,\omega_k)\in\Gamma'_-$
by (iv) of Lemma \ref{le:basic_geometric_properties:1}.
Since $\tilde{t}_k\to 0$,
we see that $(\ol{x}_k,\omega_k)\to (y,\omega)$ as $k\to\infty$.
Moreover, by construction $(\ol{x}_k,\omega_k)\in\Gamma'_-\cup\Gamma'_0$
for all $k$.

That is, $\omega_k\cdot\nu(\ol{x}_k)\leq 0$ for all $k$
and hence $\omega\cdot\nu(y)\leq 0$ in the limit as $k\to\infty$
because $\nu$ is continuous on $\partial G$.
This shows that $(y,\omega)\in\Gamma'_-\cup\Gamma'_0$.
The proof of (ii) is complete.
\end{proof}

\section{Discussion}\label{disc}

In the case where the spatial domain $G\not=\R^3$ only   solutions with limited Sobolev regularity can  be achieved (without essential restrictions on data $(f,g)$) for transport problems  and  the regularity of  solutions   does not increase according to data $f,\ g$. 
The corresponding boundary matrix (here a scalar) has a variable rank on $\Gamma$ and the boundary matrix changes the definiteness on $\Gamma$.
One knows from the general theory of first-order PDE systems that solutions of these kinds of initial boundary value problems possess only limited regularity.
In the case of initial inflow boundary value transport problems, the limit of Sobolev regularity  with respect to $x$-variable  seems to be ${3\over 2}$.
If the transport equation does not contain the slowing down term $a{\p \psi{E}}$ the regularity with respect to $E$-variable is not necessarily limited (see \cite[section 4.1.3]{tervo23}). When the term $c\Delta_S\psi$ is included (not considered in this paper) the basic existence results  (\cite[section 6]{tervo19}) guarantee that
$\psi\in L^2(G\times I,H^1(S))=H^{(0,1,0)}(G\times S\times I)$. The partially elliptic term 
$c\Delta_S\psi$ (that is, elliptic  with respect to $\omega$-variable only) actually smooths the solution and higher order regularity with respect to $\omega$-variable is  possible (this issue needs  further research). When the term $c\Delta_S\psi$ is not included  the regularity with respect to $\omega$-variable is limited as well.

Nevertheless, when one appropriately restrict the data $f,\ g$ the existence of higher order regular solutions is possible. This paper considered one such case where $g=0$ and $f_{|\Gamma_-}=0$ in the Sobolev sense. Our results in section \ref{nis} can be considered as a variant of the results given in \cite{nishitani96} where certain general symmetric first-order partial differential systems $\sum_{j=1}^n A_j(x)\partial_{x_j}u+B(x)u=f$ and the corresponding time-dependent systems are considered in an open bounded domain $\Omega\subset\R^n$. 
In \cite{nishitani96} the boundary matrix changes its definiteness only
when crossing an $(n-2)$-dimensional submanifold of the boundary. This is the situation in our case since the boundary matrix (here a scalar) is
$\sum_{j=1}^3\nu_j A_j=\omega\cdot\nu$ and $\omega\cdot\nu$ changes its sign only on $\Gamma_0'\subset \Gamma'=(\partial G)\times S=\partial(G\times S)$,
where $\Gamma'_0=\{(y,\omega)\in (\partial G) \times S\ |\ \omega\cdot \nu(y)=0\}$ is a $3$-dimensional submanifold of the $5$-dimensional $G\times S$.
Yet the problem we consider and the techniques we use are different.
Note that our model contains the partial integral operator $K_r$. Moreover, our basic domain is of the form $G\times S$, where $S$ is a manifold without boundary.
We apply explicit solution formulas and the smallest closed extensions of transport operators to deduce existence results in higher order Sobolev spaces. In 
\cite{nishitani96} one has applied cutting/weight function techniques and the basic assumption of definiteness (1.2) therein is needed to obtain relevant apriori estimates for the relevant operators. 
Moreover, the existence of solutions is based on the Hahn-Banach theorem.

The regularity of solutions is crucial in the convergence analysis of numerical schemes. Perhaps the most effective methods for transport problems are based on the finite element methods (FEM) in spatial variables.  The error estimates for various (discontinuous) FEM schemes  are also known. 
The transport equation (\ref{i1}) is hyperbolic in nature. So discontinuous FEM schemes are favourable.

There exist some special FEM convergence results for transport problems. For example, in \cite{han} one has considered mono-kinetic (mono-energetic) transport equation $\omega\cdot\nabla_x\psi+\Sigma\psi-K_r\psi=f$. Firstly one  applies the so-called discrete-ordinate scheme to discretize the angular variable $\omega$. The collision term $K_r\psi$ is discretized by using an appropriate numerical quadrature.
As a result one obtains a semi-discretized approximation, say $\psi_{do}$, which obeys a system of first order partial differential equations with respect to spatial variable $x$.
The system is further discretized with respect to $x$ by applying the discontinuous Galerkin method where the (fully discretized) approximation, say $P_h(\psi_{do})$, is the elementwise $L^2$-orthogonal projection of $\psi_{do}$ onto the finite dimensional element  space (that is,
$P_h(\psi_{do})_{|K}=P_k({\psi_{do}}_{|K})$ where $P_k:L^2(K)\to P_k(K)$ is the orthogonal projection). 
The element space $ P_k(K)$ is the set of piecewise polynomials whose degree is less than or equal to $k$ in each element $K$.
The convergence is under control if one assumes the regularity $\psi\in H^{(r,0)}(G\times S),\ r>{1\over 2}$. Actually, one gets the convergence rate $\sim h^{\min\{r,k+1\}-{1\over 2}}$, where $h$ is the maximal diameter of  elements $K$.  
For example, if $k=1$ and $r=1$ one gets the convergence rate $\sim h^{{1\over 2}}$ and if $k=1$ and $r=2$ one gets the convergence rate $\sim h^{{3\over 2}}$. Hence we find that regularity of solutions strongly affects the convergence speed.
Many other similar results can be found in the literature.
The above procedure can be applied to more general equations (\ref{i1-a}). However, the presence of terms $a{\p \psi{E}}$,
$c\Delta_S\psi$ and
$d\cdot\nabla_S\psi$ requires additional treatments.

Coupled systems of equations like (\ref{i1-a}) can be used to model simultaneous propagation of several types of particles (photons, electrons, positrons, etc.) 
in radiotherapy dose calculation (e.g. \cite{hensel}, \cite{frank08}, \cite{frank10},
\cite{tervo17}, \cite{tervo18-up}).
In practice the corresponding initial inflow boundary value problem is solved numerically. The discrete (matrix) equation is high dimensional and so only numerical methods whose rigorous convergence is under control can be used successfully.
When the solution $\psi$ is regular enough the discontinuous FEM methods offer one such approach as part of numerical schemes. Some of today's commercial dose calculation codes already use FEM approaches. So the regularity analysis of solutions is important also from that point of view.

\subsection*{Acknowledgements}
This work was supported by the Research Council of Finland (Flagship of Advanced Mathematics for Sensing Imaging and Modelling grant 358944).


\end{document}